\documentclass[reqno,11pt]{amsart}

\usepackage{amsmath,amssymb,amsthm,tikz,mathtools}
\usepackage{verbatim}
\usepackage{xcolor}
\usepackage[normalem]{ulem}
\usepackage{mathrsfs}
\usepackage{cancel}

\usepackage{hyperref}
\usepackage[nameinlink]{cleveref}

\DeclareMathOperator*{\ess}{ess\,}

\newtheorem{theorem}{Theorem}[section]
\newtheorem{example}{Example}[section]
\newtheorem{definition}{Definition}[section]
\newtheorem{lemma}{Lemma}[section]
\newtheorem{corollary}{Corollary}[section]
\newtheorem{proposition}{Proposition}[section]

\newtheorem{remark}{Remark}[section]

\newcommand\dist{\mathrm{dist}\,}

\newcommand{\R}{\mathbb R}

\newcommand{\sing}{\text{Sing}}
\newcommand{\reg}{\text{Reg}}

\numberwithin{equation}{section}

\newcommand{\intav}[1]{\mathchoice {\mathop{\vrule width 6pt height 3 pt depth  -2.5pt
\kern -8pt \intop}\nolimits_{\kern -6pt#1}} {\mathop{\vrule width
5pt height 3  pt depth -2.6pt \kern -6pt \intop}\nolimits_{#1}}
{\mathop{\vrule width 5pt height 3 pt depth -2.6pt \kern -6pt
\intop}\nolimits_{#1}} {\mathop{\vrule width 5pt height 3 pt depth
-2.6pt \kern -6pt \intop}\nolimits_{#1}}}

\title[On FBP shaped by varying singularities]{On free boundary problems shaped by\\ varying singularities}

\author[D.J. Ara\'ujo]{Dami\~ao J. Ara\'ujo}
\address{Department of Mathematics, Universidade Federal da Para\'iba, 58059-900, Jo\~ao Pessoa-PB, Brazil}{}
\email{araujo@mat.ufpb.br}

\author[A. Sobral]{Aelson Sobral}
\address{Applied Mathematics and Computational Sciences (AMCS), Compu\-ter, Electrical and Mathematical Sciences and Engineering Division (CEMSE), King Abdullah University of Science and Technology (KAUST), Thuwal, 23955 -6900, Kingdom of Saudi Arabia}{} 
\email{aelson.sobral@kaust.edu.sa} 

\author[E. V. Teixeira]{Eduardo V. Teixeira}
\address{Department of Mathematics, Oklahoma State University, 401 Mathematical
Sciences, Stillwater, OK 74078, USA}
\email{eduardo.teixeira@okstate.edu}

\author[J.M.~Urbano]{Jos\'{e} Miguel Urbano}
\address{Applied Mathematics and Computational Sciences (AMCS), Compu\-ter, Electrical and Mathematical Sciences and Engineering Division (CEMSE), King Abdullah University of Science and Technology (KAUST), Thuwal, 23955 -6900, Kingdom of Saudi Arabia and CMUC, Department of Mathematics, University of Coimbra, 3000-143 Coimbra, Portugal}{} 
\email{miguel.urbano@kaust.edu.sa}

\begin{document}



\begin{abstract} 
We start the investigation of free boundary variational models featuring varying singularities. The theory depends strongly on the nature of the singular power $\gamma(x)$ and how it changes. Under a mild continuity assumption on $\gamma(x)$, we prove the optimal regularity of minimizers. Such estimates vary point-by-point, leading to a continuum of free boundary geometries. We also conduct an extensive analysis of the free boundary shaped by the singularities.  Utilizing a new monotonicity formula, we show that if the singular power $\gamma(x)$ varies in a $W^{1,n^{+}}$ fashion, then the free boundary is locally a $C^{1,\delta}$ surface, up to a negligible singular set of Hausdorff co-dimension at least $2$.
\end{abstract}  

\date{\today}

\maketitle

\tableofcontents

\section{Introduction} \label{sct intro}

We develop a variational framework for the analysis of free boundary problems that include a continuum of singularities. The mathematical setup leads to the minimization of an energy-functional of the type
\begin{equation}\label{general F intro}  
    \mathscr{E}(v, \mathcal{O}) = \int_{\mathcal{O}} {F}(Dv, v, x) \, dx,
\end{equation}
whose Lagrangian, $F(\vec{p}, v, x)$, is non-differentiable with respect to the $v$ argument, and the degree of singularity varies with respect to the spatial variable $x$. The singularity variation exerts an intricate influence on the free boundary's trace and shape in a notably unpredictable manner. This dynamic not only alters the geometric behaviour of the solution but also significantly impacts the regularity of the free boundary. As a consequence, the associated Euler-Lagrange equation gives rise to a rich new class of singular elliptic partial differential equations, which, in their own right, present an array of intriguing and independent mathematical challenges and interests.

Singular elliptic PDEs, particularly those involving free boundaries, find applications in a variety of fields, including thin film flows, image segmentation, shape optimization, and biological invasion models in ecology, to cite just a few. Mathematically, such models lead to the analysis of an elliptic PDE of the form
\begin{equation}\label{sing PDE intro}    
    \Delta u = \mathfrak{s}(x,u)\chi_{\{u>0 \}},  
\end{equation}
within a domain $\Omega \subset \mathbb{R}^n$. The defining characteristic of the PDE above lies in the singular term $\mathfrak{s}\colon \Omega \times (0,\infty) \to \mathbb{R}$, which becomes arbitrarily large near the zero level set of the solution, \textit{i.e.}, 
\begin{equation}\label{papoca mininu}
    \lim\limits_{v\to 0} \mathfrak{s}(x,v) = \infty.
\end{equation}
Fine regularity properties of solutions to \eqref{sing PDE intro}, along with geometric measure estimates and eventually the differentiability of their free boundaries, $\partial \{u> 0 \}$, are inherently intertwined with {\it quantitative} information concerning the blow-up rate outlined in \eqref{papoca mininu}. Heuristically, solutions of PDEs with a faster singular blow-up rate will exhibit reduced regularity along their free boundaries. Existing methods for treating these singular PDE models, in various forms, rely to some extent on the {\it uniformity} of the blow-up rate prescribed in \eqref{papoca mininu}. 

In this paper, we investigate a broader class of variational free boundary problems, extending our focus to encompass varying blow-up rates. That is, we are interested in PDE models involving singular terms with fluctuating asymptotic behavior, 
\begin{equation}\label{intro motivacao eq1}
    \Delta u  \sim u^{-p(x)},
\end{equation}
for some function $p\colon \Omega \to [0,1)$. As anticipated, the analysis will be variational, \textit{i.e.}, we will investigate local minimizers of a given non-differentiable functional, as described in \eqref{general F intro}, which exhibit a spectrum of varying exponents of non-differentiability. 

The investigation of the static case, \textit{i.e.}, of PDE models in the form of $\Delta u \sim u^{-p_0}$, where $0 < p_0 < 1$, has a rich historical lineage, tracing its roots to the classical Alt--Phillips problem, as documented in \cite{AP, P1, P2}. This elegant problem has served as a source of inspiration, sparking significant advancements in the domain of free boundary problems, as exemplified by works like \cite{AT, DSS, DKV, ES, ST, Teix, WY, Y}, to cite just a few. Remarkably, the Alt--Phillips model serves as a bridge connecting the classical obstacle problem, which pertains to the case $p_0 = 0$, and the cavitation problem, achieved as the limit when $p_0 \nearrow 1$. Each intermediary model exhibits its own unique geometry. That is, solutions present a precise geometric behavior at a free boundary point, viz. $u \sim \text{dist}^{\beta}(x, \partial \{u> 0\})$, for a critical, well-defined and uniform exponent $\beta(p_0)$. 

Mathematically, the variation of the singular exponent brings several new challenges, as the model prescribes multiple free boundary geometries. The main difficulty in analyzing free boundary problems with varying singularities relies on quantifying how the local free boundary geometry fluctuations affect the regularity of the solution $u$ as well as the behavior of its associated free boundary $\partial \{u>0\}$. In essence, the main quest in this paper is to understand how changes in the free boundary geometry directly influence its local behaviour.  

From the applied viewpoint, the model studied here reflects the heterogeneity of external factors that govern reaction rates within porous catalytic regions where the gas density $u(x)$ is distributed (see, \textit{e.g.}, \cite{Aris}). For instance, in \cite{BN}, the static model $\Delta u \sim u^{-p_0}$ is derived as a singular limit of nonlinear eigenvalue problems motivated by the Langmuir–Hinshelwood principle. Yet in realistic settings—such as heterogeneous catalytic surfaces, thin films on patterned substrates, or porous composite materials—the effective singularity strength is rarely uniform. Instead, it varies across space due to local fluctuations in adsorption, microstructure, or wettability. Multiscale analyses under these conditions show that such heterogeneities fundamentally alter the limiting behavior, and the appropriate macroscopic description is no longer a constant-exponent model but rather a singular PDE with spatially varying singularities, as in \eqref{intro motivacao eq1}.

In this inaugural paper, our focus is directed toward fine regularity properties of local minimizers of the energy-functional
\begin{equation}\label{intro fifi}
    J_{\delta(x)}^{\gamma(x)}(v) \coloneqq  \int \frac{1}{2} \left| Dv \right|^2 + \delta(x) (v^+)^{\gamma(x)} dx,
\end{equation}
where the functions $\gamma(x)$ and $\delta(x)$ possess specific properties that will be elaborated upon in due course. In connection with the theory of singular elliptic PDEs, minimizers of \eqref{intro fifi} are distributional solutions of 
$$\left\{
    \begin{array}{rllcl}
        \Delta u & = & \delta(x) \gamma (x) u^{\gamma(x)-1} & \mathrm{in} & \,\{u>0\}\\
     Du & = & 0 & \mathrm{on} & \partial \{u>0\},
    \end{array}
\right.
$$
with the free boundary condition being observed by local regularity estimates, to be shown in this paper. 

The paper is organized as follows. In Section \ref{prelim-sect}, we discuss the mathematical setup of the problem and the scaling feature of the energy-functional \eqref{fifi}. We also establish the existence of minimizers as well as local $C^{1,\alpha_\ast}$-regularity, for some $0<\alpha_\ast<1$, independent of the modulus of continuity of $\gamma(x)$. The final preliminary result in Section \ref{prelim-sect} concerns non-degeneracy estimates. In Section \ref{sct-grad}, we obtain gradient estimates near the free boundary, quantifying the magnitude of $Du(y)$ in terms of the pointwise value $u(y)$. We highlight that the results established in Sections \ref{prelim-sect} and \ref{sct-grad} are all independent of the continuity of $\gamma(x)$. However, when $\gamma(x)$ varies randomly, regularity estimates of $u$ and its non-degeneracy properties along the free boundary have different homogeneities, and thus no further regularity properties of the free boundary are expected to hold. We tackle this issue in Section \ref{sct beicodebode}, where under a very weak condition on the modulus of continuity of $\gamma(x)$, we establish sharp {\it pointwise} growth estimates of $u$. The estimates from Section \ref{sct beicodebode} imply that near a free boundary point $x_0 \in \partial \{u>0 \}$, the minimizer $u$ behaves {\it precisely} as $\sim d^{\frac{2}{2-\gamma(x_0)}}$, with universal estimates. Section \ref{sct Hausdorff} is devoted to Hausdorff estimates of the free boundary. In Section \ref{sct monotonicity}, we obtain a Weiss-type monotonicity formula which yields blow-up classification, and in Section \ref{sct FB reg}, we discuss the regularity of the free boundary $\partial \{u>0 \}$. 

We conclude this introduction by emphasizing that the complexities inherent in the dynamic singularities model extend far beyond the boundaries of the specific problem under consideration in this study. The challenges posed by the program put forward in this paper call for the development of new methods and tools. We are optimistic that the solutions crafted in this research can have a broader impact, proving invaluable in the analysis of a wide range of mathematical problems where similar intricacies and complexities manifest themselves.

\section{Preliminary results}\label{prelim-sect}

\subsection{Mathematical setup} 

We start by describing precisely the mathematical setup of our problem. We assume $\Omega \subset \mathbb{R}^n$ is a bounded smooth domain and $\delta, \gamma \colon \Omega \to \R_0^+$ are bounded measurable functions. 

For each subset $\mathcal{O} \subset \Omega$, we denote
\begin{equation}\label{home}
\gamma_\ast (\mathcal{O}) \coloneqq  \mathop{\ess \inf}_{y\in \mathcal{O}} \, \gamma (y) \ \quad \text{and} \quad \, \gamma^\ast (\mathcal{O}) \coloneqq  \mathop{\ess \sup}_{y \in \mathcal{O}} \, \gamma (y).  
\end{equation}
In the case of balls, we adopt the simplified notation
\begin{equation*}
    \gamma_\ast (x,r) \coloneqq \gamma_\ast (B_r(x)) \qquad \text{and} \qquad \gamma^\ast(x,r) \coloneqq \gamma^\ast (B_r(x)).
\end{equation*}

Throughout the whole paper, we shall assume
\begin{equation} \label{H1}
0<\gamma_\ast(\Omega)   \le  \gamma^\ast(\Omega)  \le 1.
\end{equation}

For a non-negative boundary datum $0 \leq \varphi \in H^1(\Omega) \cap L^\infty(\Omega)$, we consider the problem of minimizing the functional 
\begin{equation}\label{fifi}
	\mathcal{J}^\delta_\gamma(v, \Omega) \coloneqq  \int_{\Omega} \frac{1}{2} \left| Dv \right|^2 + \delta(x) (v^+)^{\gamma(x)} dx
\end{equation}
among competing functions 
$$
v \in {\mathcal A} \coloneqq  \left\{ v \in H^1(\Omega) \ :\ v-\varphi \in H_0^1(\Omega) \right\}. 
$$
We say $u \in \mathcal{A}$ is a minimizer of \eqref{fifi} if 
$$
\mathcal{J}^\delta_\gamma(u, \Omega) \leq \mathcal{J}^\delta_\gamma(v, \Omega), \quad \forall v \in \mathcal{A}. 
$$
Note that minimizers as above are, in particular, local minimizers in the sense that, for any open subset $\Omega^\prime \subset \Omega$,
$$
\mathcal{J}^\delta_\gamma(u, \Omega^\prime) \leq \mathcal{J}^\delta_\gamma(v, \Omega^\prime), \quad \forall v \in H^1(\Omega^\prime)\ :\ v-u \in H_0^1(\Omega^\prime).
$$

\subsection{Scaling}

Some of the arguments used recurrently in this paper rely on a scaling feature of the functional \eqref{fifi} that we detail in the sequel for future reference. Let $x_0 \in \Omega$ and consider two parameters $A,B \in (0,1]$. If $u\in H^1(\Omega)$ is a minimizer of $\mathcal{J}_\gamma^\delta(v,B_A(x_0))$, then
\begin{equation}\label{august}
w (x) \coloneqq  \frac{u(x_0 + Ax)}{B}, \quad x\in B_1
\end{equation}
is a minimizer of the functional
$${\mathcal{J}}^{\tilde{\delta}}_{\tilde{\gamma}}(v,B_1) \coloneqq  \int_{B_1} \frac{1}{2} \left| Dv \right|^2 + \tilde{\delta}(x) v^{\tilde{\gamma}(x)} dx,$$
with
$$
\tilde{\delta}(x) \coloneqq  B^{\gamma(x_0 + Ax)}\left(\frac{A}{B} \right)^2  \delta(x_0 + Ax)  \quad \mathrm{and} \quad \tilde{\gamma}(x) \coloneqq  \gamma(x_0 + Ax).
$$   
Indeed, by changing variables,
$$\int_{B_A(x_0)} \frac{1}{2} \left| Du (x)\right|^2 + \delta(x) u(x)^{\gamma(x)} dx  $$
\begin{eqnarray*}
 &= &  A^{n} \int_{B_1} \frac{1}{2} \left| Du(x_0 + Ax) \right|^2 + \delta(x_0 + Ax)u(x_0 + Ax)^{\gamma(x_0 + Ax)}\, dx\\
 &   = &  A^n \int_{B_1} \frac{1}{2} \left| \left(\frac{B}{A} \right) Dw (x) \right|^2 + \delta(x_0 + Ax) \left[ B w (x) \right]^{\gamma(x_0 + Ax)} dx \\
 &   = & A^{n-2}B^2 \int_{B_1} \frac{1}{2} \left| Dw (x) \right|^2 + \frac{A^2}{B^{ 2-\gamma(x_0+Ax)}}  \delta(x_0 + Ax) \left[w (x) \right]^{\gamma(x_0 + Ax)} dx \\
 &   = & A^{n-2}B^2 \int_{B_1} \frac{1}{2} \left| Dw (x) \right|^2 + \tilde{\delta}(x)\left[ w (x) \right]^{\tilde{\gamma}(x)} dx.
\end{eqnarray*}
Observe that since $0<B \leq 1$, $\tilde{\delta}$ satisfies  
$$
\|\tilde{\delta}\|_{L^\infty(B_1)} \leq B^{\gamma_\ast(x_0,A) - 2} A^2\| \delta \|_{L^\infty(B_A(x_0))}. 
$$
In particular, choosing $A=r$ and $B=r^\beta$, with $0<r\leq 1$ and 
$$
\beta = \frac{2}{2-\gamma_\ast(x_0,A)},
$$
we obtain $\|\tilde \delta\|_{L^\infty(B_1)} \leq \|\delta\|_{L^\infty(B_r(x_0))}$.

\subsection{Existence of minimizers}
We start by proving the existence of non-negative minimizers of the functional \eqref{fifi} and deriving global $L^\infty$-bounds. 

\begin{proposition} Under the conditions above, namely \eqref{H1}, there exists a minimizer $u \in {\mathcal A} $ of the energy-functional \eqref{fifi}. Furthermore, $u$ is non-negative in $\Omega$ and $\|u\|_{L^\infty(\Omega)} \le \|\varphi\|_{L^\infty(\Omega)}$.
\end{proposition}

\begin{proof}
Let 
$$m = \inf_{v\in {\mathcal A} } \mathcal{J}^\delta_\gamma(v, \Omega)$$
and choose a minimizing sequence $u_k \in {\mathcal A} $ such that, as $k\to \infty$,
$$\mathcal{J}^\delta_\gamma(u_k, \Omega) \longrightarrow m.$$
Then, for $k \gg 1$, we have
\begin{eqnarray*}
 \left\| Du_k \right\|_{L^2(\Omega)}^2 & = &  2\mathcal{J}^\delta_\gamma(u_k, \Omega) - 2 \int_{\Omega} \delta(x) (u_k^+)^{\gamma(x)} dx \\
 & \leq & 2(m+1).
\end{eqnarray*}
From Poincar\'e inequality, we also have
\begin{eqnarray}
\left\|u_k \right\|_{L^2(\Omega)}&  \leq & \left\| u_k - \varphi \right\|_{L^2(\Omega)} + \left\| \varphi \right\|_{L^2(\Omega)} \qquad \nonumber \\
& \leq & C \left\| Du_k - D\varphi \right\|_{L^2(\Omega)} + \left\| \varphi \right\|_{L^2(\Omega)} \qquad \nonumber \\
& \leq & C \left\| Du_k \right\|_{L^2(\Omega)}  +  C \left\|  D\varphi \right\|_{L^2(\Omega)} + \left\| \varphi \right\|_{L^2(\Omega)}, \qquad \nonumber
\end{eqnarray}
and so
\begin{equation}\label{ramp}
\left\|u_k \right\|_{L^2(\Omega)} \leq C\left(m + 1 + \|\varphi\|_{H^1(\Omega)}\right),    
\end{equation}
for some dimensional constant $C>0$, which implies $\{u_k\}_k$ is bounded in $H^1(\Omega)$. Consequently, for a subsequence (relabelled for convenience) and a function $u \in H^1 (\Omega)$, we have
$$u_k \longrightarrow u,$$
weakly in $H^1 (\Omega)$, strongly in $L^2 (\Omega)$ and pointwise for a.e. $x \in \Omega$. Using Mazur's theorem, it is standard to conclude that $u \in {\mathcal A} $.

The weak lower semi-continuity of the norm gives 
$$\int_{\Omega} \frac{1}{2} \left| Du \right|^2 dx \leq \liminf_{k\to \infty} \int_{\Omega} \frac{1}{2} \left| Du_k \right|^2 dx$$
and the pointwise convergence and Lebesgue's dominated convergence give
$$\int_{\Omega} \delta(x) (u_k^+)^{\gamma(x)} dx \longrightarrow \int_{\Omega} \delta(x) (u^+)^{\gamma(x)} dx.$$
We conclude that
$$\mathcal{J}^\delta_\gamma(u, \Omega)  \leq  \liminf_{k\to \infty} \mathcal{J}^\delta_\gamma(u_k, \Omega) = m,$$
and so $u$ is a minimizer.

We now turn to the bounds on the minimizer. That $u$ is non-negative for a non-negative boundary datum is trivial since $(u^+)^{+} = u^+$, and testing the functional against $u^+ \in \mathcal{A}$ immediately gives the result.
For the upper bound, test the functional with $v=\min \left\{u, \|\varphi\|_{L^\infty(\Omega)} \right\} \in  {\mathcal A} $ to get, by the minimality of $u$, 
\begin{eqnarray*}
0 \leq \int_{\Omega} \left| D (u-v) \right|^2 dx& = &  \int_{\Omega \cap \{ u> \| \varphi \|_{L^\infty(\Omega)} \} } \left| Du \right|^2 dx\\
& = & \int_{\Omega} \left| Du \right|^2 -  \left| Dv \right|^2 dx \\
& \leq & 2 \int_{\Omega}   \delta(x)  \left[ (v^+)^{\gamma(x)} -  (u^+)^{\gamma(x)} \right]dx\\
& \leq & 0.
\end{eqnarray*}
We conclude that $v=u$ in $\Omega$ and thus $\|u\|_{L^\infty(\Omega)} \le \|\varphi\|_{L^\infty(\Omega)}$.
\end{proof}

\begin{remark} If the boundary datum $\varphi$ changes sign, the existence theorem above still applies, but the minimizer is no longer non-negative. Uniqueness may, in general, fail, even in the case of $\gamma \equiv \gamma_0 <1$.
\end{remark}

\subsection{Local gradient regularity estimates}
Our first main regularity result yields local $C^{1,\alpha}-$regularity estimates for local minimizers of \eqref{fifi}, under no further assumption on $\gamma(x)$ other than \eqref{H1}.

\begin{theorem}\label{localregthm}
Let $u$ be a minimizer of the energy-functional \eqref{fifi} under Assumption \eqref{H1}. For each subdomain $\Omega' \Subset \Omega$, there exists a constant $C>0$, depending only on the bounds on $\delta$, $n$, $\gamma_\ast(\Omega')$, $\dist (\Omega',\partial \Omega)$ and $\|u\|_{\infty}$, such that
$$
\|u\|_{C^{1,\alpha}(\Omega')} \leq C,
$$
for $\alpha = \dfrac{\gamma_\ast(\Omega')}{2-\gamma_\ast(\Omega')}$.
\end{theorem}  

For the proof of Theorem \ref{localregthm}, we will argue along the lines of \cite{GG, LQT}, but several adjustments are needed, and we will mainly comment on those. We start by noting that, without loss of generality, one can assume that the minimizer satisfies the bound
\begin{equation}\label{flatsea}
	\|u\|_{L^\infty(\Omega)} \leq 1.
\end{equation}
Indeed, $u$ minimizes \eqref{fifi} if, and only if, the auxiliary function 
$$
\overline{u}(x)\coloneqq \frac{u(x)}{M},
$$
minimizes the functional
\begin{equation}\nonumber
v \mapsto \int_{\Omega} \frac12 \left| Dv \right|^2 + \overline{\delta}(x) (v^+)^{\gamma(x)}  \, dx,
\end{equation}
where 
$$
	\overline{\delta}(x)\coloneqq M^{\gamma(x)-2}\delta(x).
$$
Taking $M=\max\{1, \|u\|_{L^\infty(\Omega)}\},$  places the new function $\overline{u}$ under condition \eqref{flatsea}; any regularity estimate proven for $\overline{u}$ automatically translates to $u$. From now on, we will always assume minimizers are normalized.

\medskip

Next, we gather some useful estimates, which can be found in \cite[Lemma 2.4 and Lemma 4.1, respectively]{LQT}. We adjust the statements of the lemmata to fit the setup treated here. Given a ball $B_R(x_0) \Subset \Omega$, we denote the harmonic replacement (or lifting) of $u$ in $B_R(x_0)$ by $h$, \textit{i.e.}, $h$ is the solution of the boundary value problem
$$
	\Delta h=0 \; \text{ in }\; B_R(x_0)  \quad \text{ and } \quad h-u \in H^1_0(B_R(x_0)).
$$ 
By the maximum principle, we have $h\geq 0$ and
\begin{equation}\label{harmonic}
\|h\|_{L^\infty(B_R(x_0))} \leq \|u\|_{L^{\infty}(B_R(x_0))}.
\end{equation}

\begin{lemma}\label{manaira}
Let $\psi \in H^1(B_R)$ and $h$ be the harmonic replacement of $\psi$ in $B_R$. There holds 
\begin{equation}\label{ICTP}
\int\limits_{B_R}|D\psi-Dh|^2 \, dx = \int\limits_{B_R}|D\psi|^2-|Dh|^2 \, dx. 
\end{equation}
\end{lemma}

\begin{lemma}\label{cariri}
Let $\psi \in H^1(B_R)$ and $h$ be the harmonic replacement of $\psi$ in $B_R$. Given $\beta \in (0,1)$, there exists $C$, depending only on $n$ and $\beta$, such that 
\begin{eqnarray*}
\int\limits_{B_r}|D\psi-(D\psi)_r|^2 \, dx & \leq & C\left( \frac{r}{R}\right)^{n+2\beta} \int\limits_{B_R}|D\psi-(D\psi)_R|^2 \, dx \\
& & + C\int\limits_{B_R}|D\psi-Dh|^2 \, dx,
\end{eqnarray*}
for each $0<r \leq R$, where 
\begin{equation*}
    (D\psi)_s \coloneqq \intav{B_s}D\psi\,dx.
\end{equation*}
\end{lemma}

We are ready to prove the local regularity result.

\begin{proof}[Proof of Theorem \ref{localregthm}] We prove the result for the case of balls $B_R(x_0) \Subset \Omega$.
Without loss of generality, assume $x_0=0$ and denote $B_R\coloneqq B_R(0)$. 
Since $u$ is a local minimizer, by testing \eqref{fifi} against its harmonic replacement, we obtain the inequality
\begin{equation} \label{fermi}
\displaystyle\int\limits_{B_R}|Du|^2-|Dh|^2 \, dx  \leq  2 \displaystyle\int\limits_{B_R} \delta(x)\left( h(x)^{\gamma(x)}-u(x)^{\gamma(x)} \right) \, dx.
\end{equation}
Next, we observe that
$$
	h(x)^{\gamma(x)}-u(x)^{\gamma(x)} \le |u(x)-h(x)|^{\gamma(x)},
$$
which is a consequence of the fact that the function 
$$
	f(t) \coloneqq (t-1)^{\gamma(x)} - t^{\gamma(x)} -1
$$ 
is decreasing for $t>1$. Using \eqref{H1}, together with \eqref{flatsea} and \eqref{harmonic}, we get
\begin{equation}\label{sangiusto}
|u(x)-h(x)|^{\gamma(x)} \leq |u(x)-h(x)|^{\gamma_\ast (0,R)}, \quad \text{a.e. in } \; B_R.
\end{equation}
This readily leads to 
$$\displaystyle\int\limits_{B_R} \delta(x)\left( h(x)^{\gamma(x)}-u(x)^{\gamma(x)} \right) \, dx \leq \|\delta\|_{L^\infty(\Omega)} \int\limits_{B_R} |u(x)-h(x)|^{\gamma_\ast (0,R)}\, dx.$$	
In addition, by combining H\"older  and Sobolev inequalities, we obtain
\begin{eqnarray}
    \int\limits_{B_R} |u-h|^{\gamma_\ast (0,R)}\, dx & \leq & C |B_R|^{1-\frac{\gamma_\ast (0,R)}{2^\ast}}\left( \,\int\limits_{B_R} |u-h|^{2^\ast}\, dx \, \right)^{\frac{\gamma_\ast (0,R)}{2^\ast}} \nonumber \\
& \leq & C |B_R|^{1-\frac{\gamma_\ast (0,R)}{2^\ast}}\left( \displaystyle \int\limits_{B_R} |Du-Dh|^{2}\, dx \right)^{\frac{\gamma_\ast (0,R)}{2}} \label{trieste}
\end{eqnarray}
for $2^\ast=\dfrac{2n}{n-2}$. 

Therefore, using Lemma \ref{manaira}, together with \eqref{fermi}, \eqref{sangiusto} and \eqref{trieste}, we get
\begin{equation}\label{Cariri7}
\int\limits_{B_R}|Du-Dh|^2 \, dx \leq C |B_R|^{\frac{2(2^\ast-\gamma_\ast (0,R))}{2^\ast(2-\gamma_\ast (0,R))}}=CR^{n+2\frac{\gamma_\ast (0,R)}{2-\gamma_\ast (0,R)}}.
\end{equation}
Finally, by taking 
$$
\beta=\frac{\gamma_\ast (0,R)}{2-\gamma_\ast (0,R)} \in (0,1),
$$
in Lemma \ref{cariri}, we conclude 
$$\int\limits_{B_r}|Du-(Du)_r|^2 \, dx $$
$$\leq C\left( \frac{r}{R}\right)^{n+2\frac{\gamma_\ast (0,R)}{2-\gamma_\ast (0,R)}} \int\limits_{B_R}|Du-(Du)_R|^2 \, dx + CR^{n+2\frac{\gamma_\ast (0,R)}{2-\gamma_\ast (0,R)}},
$$
for each $0<r \leq R$. Campanato's embedding theorem completes the proof, see for instance \cite[Theorem 2.9]{EG} and \cite[Theorem 3.1 and Lemma 3.4]{QF}.
\end{proof}

Hereafter, in this paper, we assume $\Omega = B_1 \subset \mathbb{R}^n$ and, according to what was argued around \eqref{flatsea}, fix a normalized, non-negative minimizer, $0\le u \le 1$, of the energy-functional \eqref{fifi}. 

\begin{remark} It is worth noting that the proof of Theorem \ref{localregthm} does not rely on the non-negativity property of $u$. Therefore, the same conclusion applies to the two-phase problem, and the proof remains unchanged.

\end{remark}

\subsection{Non-degeneracy}

We now turn our attention to local non-degeneracy estimates. We will assume $\delta(x)$ is bounded below away from zero, namely that it satisfies the condition
\begin{equation}\label{bound delta}
	 \mathop{\ess \inf}_{x\in B_1} \delta(x) \eqqcolon \delta_0>0. 
\end{equation} 

\begin{theorem}\label{nondeg}
Assume \eqref{bound delta} is in force. For any $y \in \overline{\{u>0\}}\cap B_{1/2}$ and $0<r < 1/2$, we have
\begin{equation}\label{nondeg_est}
	\sup_{\partial B_r(y)} u \geq c \, r^{\frac{2}{2-\gamma^\ast(y,r)}},
\end{equation}
where $c>0$ depends only on $n$, $\delta_0$ and $\gamma_\ast (0,1)$.
\end{theorem}

\begin{proof}
With $y \in \{u>0\}$ and $0 < r < 1/2$ fixed, define the auxiliary function $\varphi$ by
$$
	\varphi(x)\coloneqq u(x)^{2-\gamma^\ast(y,r)}-c|x-y|^2,
$$
for a constant $c>0$ satisfying
$$
	0 < c \leq \min \left\{ 1, \frac{\delta_0 \gamma_\ast (0,1) }{2n} \right\}.
$$
Note that $\Delta \varphi \geq 0$ in the weak sense in $\{u>0\} \cap B_r(y)$. Indeed, for any $B \Subset \{u>0\} \cap B_r(y)$ and $v \in H^1_0(B)$, we use $u^{1-\gamma^{\ast}(y,r)} v \in H^1_0(B)$ as a test function for $\Delta u = \delta(x)\gamma(x)u^{\gamma(x)-1}$ in the weak sense to get that the quantity
\begin{eqnarray*}
	\mathcal{I} = \int_B D\varphi \cdot Dv\,dx
\end{eqnarray*}
satisfies
\begin{eqnarray*}
	\mathcal{I} & = & (2-\gamma^\ast(y,r))\int_B u^{1-\gamma^\ast(y,r)} Du \cdot Dv\,dx - 2c \int_B (x-y) \cdot Dv\, dx\\
	 &= & (2-\gamma^\ast(y,r))\int_B Du \cdot D\left(u^{1-\gamma^\ast(y,r)} v\right)\,dx + 2cn \int_B v\, dx\\
	 & &  -(2-\gamma^\ast(y,r))\int_B(1-\gamma^{\ast}(y,r)) u^{-\gamma^{\ast}(y,r)}|Du|^2 v\,dx
	 \\
	  &\leq & -(2-\gamma^\ast(y,r))\int_B \delta(x)\gamma(x) v \,dx\  + 2cn \int_B v\, dx\\
	  &\leq & \int_B \left(-\delta_0 \gamma_{\ast}(0,1) + 2cn\right) v\, dx \leq 0,
\end{eqnarray*}
where the last inequality follows from the choice of $c$. In addition, since $\varphi(y)>0$, by the Maximum Principle, 
$$
	 \partial\left(\{u>0\} \cap B_r(y)\right) \cap \{\varphi>0\} \not = \emptyset.
$$
Consequently, since ${\frac{1}{2-\gamma^\ast(y,r)}} \leq 1$
$$
    \sup_{\partial B_r(y)} u > c^{{\frac{1}{2-\gamma^\ast(y,r)}}} \, r^{{\frac{2}{2-\gamma^\ast(y,r)}}}\geq c \, r^{{\frac{2}{2-\gamma^\ast(y,r)}}},
$$
and the proof is complete for $y \in \{u>0\}$; the general case follows by continuity.
\end{proof}

\section{Gradient estimates near the free boundary} \label{sct-grad}

In this section, we study gradient oscillation estimates for minimizers of \eqref{fifi} in regions relatively close to the free boundary. We first show that pointwise flatness implies an $ L^\infty$-estimate.

\begin{lemma}\label{april.2} 
Let $u$ be a local minimizer of the energy-functional \eqref{fifi} in $B_1$. Assume that
$$
    \gamma_\ast (0,1)>0.
$$
There exists a constant $C>4$, depending only on $\gamma_\ast (0,1)$ and universal parameters, such that, if 
$$
    u(x) \le \frac{1}{C} \, r^{\frac{2}{2-\gamma_\ast(0,1)}}, 
$$
for $x\in B_{1/2}$ and $r \leq 1/4$, then
\begin{equation}\label{conclusion magic lemma}
    \sup_{B_r(x)} u \leq C r^{\frac{2}{2-\gamma_\ast(0,1)}}.
\end{equation}   
\end{lemma}

\begin{proof}
Fix $0<r \leq 1/4$ and consider $j_r \in \mathbb{N}$ such that
$$
    2^{-(j_r+1)} \leq r < 2^{-j_r}.
$$
For $j \in \{1, 2, \cdots, j_r \}$, define
\[
S_j(x,u) \coloneqq \sup_{B_{2^{-j}}(x)} u,
\qquad
a_j \coloneqq 2^{j \beta_\ast(0,1)} S_j(x,u),
\]
and for $j=j_r+1$,
\[
a_{j_r+1} \coloneqq r^{-\beta_\ast(0,1)} \sup_{B_r(x)} u,
\]
where
$$
    \beta_\ast(0,1) = \frac{2}{2-\gamma_\ast (0,1)}.
$$
Then, to obtain \eqref{conclusion magic lemma}, it is enough to prove that there exists a constant $C>4$ such that
\begin{equation}\label{recursive arg shaggy}
    a_{j+1} \leq \max\{C,a_j\}, \quad \forall j \in \{1, 2, \cdots, j_r\}.
\end{equation}
Indeed, if this is true, then by recurrence
$$
    a_j \leq \max\{C,a_1\}, \quad \forall j \in \{1, 2, \cdots, j_r+1\},
$$
and so
$$
    a_{j_r + 1} = r^{-\beta_\ast(0,1)}\sup_{B_r(x)}u \leq \max\{C,a_1\} = C,
$$
where the last equality follows from the fact that, since $u$ is normalized, $a_1 = 2^{\beta_\ast} \sup_{B_{1/2}(x)}u \leq 4$ and $C>4$. Let us now suppose, seeking a contradiction, that \eqref{recursive arg shaggy} fails. Then, for each integer $k>0$, there exist a minimizer $u_k$ of \eqref{fifi} in $B_1$, $x_k \in B_{1/2}$ and $ 0< r_k < 1/4$, such that 
$$
    u_k(x_k) \le \frac{1}{k} \, r_k^{\beta_\ast(0,1)},
$$
but
\begin{equation}\label{equiv inequality}
    a_{j_k + 1} > \max\{k,a_{j_k}\}, \qquad \text{for some} \ j_k \in \{1,2,\cdots, j_{r_k}\}.
\end{equation}
In the sequel, define
\[
\varphi_k(x) \coloneqq 
   \frac{u_k(x_k+2^{-j_k}x)}{S_{j_k+1}(x_k,u_k)},
   \qquad x\in B_1.
\]
This function satisfies
\begin{equation}\label{this better work}
    \sup_{B_1} \varphi_k \leq 4, \qquad \sup_{B_{1/2}} \varphi_k = 1, \qquad \text{and} \quad \varphi_k(0) = O(k^{-2}).
\end{equation}
Indeed, from \eqref{equiv inequality}, we obtain
\[
\sup_{B_1} \varphi_k 
   = \frac{S_{j_k}(x_k,u_k)}{S_{j_k + 1}(x_k,u_k)}
   < \frac{2^{(j_k+1)\beta_\ast(0,1)}}{2^{j_k\beta_\ast(0,1)}} \leq 4.
\]
From scaling, it directly follows that $\sup_{B_{1/2}} \varphi_k = 1$, and finally,
$$
    \varphi_k(0) \leq \frac{1}{k^2}\,\frac{r_k^{\beta_\ast(0,1)}}{2^{(j_k+1)\beta_\ast(0,1)}} \leq \frac{1}{k^2}.
$$
In addition, note that $\varphi_k$ minimizes
$$
    v \longmapsto \int_{B_1} \frac{1}{2} \left| Dv \right|^2 + \delta_k(x) v^{\gamma_k(x)}\, dx,
$$
for
$$
    \delta_k(x)\coloneqq  \delta(x_k+2^{-j_k}x)\frac{2^{-2j_k}}{s_k^{2-\gamma(x_k+2^{-j_k}x)}} \quad \text{and} \quad \gamma_k(x)\coloneqq \gamma(x_k+2^{-j_k}x),
$$ 
where
$$
    s_k \coloneqq S_{j_k + 1}(x_k,u_k).
$$
From \eqref{equiv inequality}, we obtain
\begin{eqnarray*}
s_k^{\gamma(x_k+2^{-j_k}x)-2}2^{-2j_k} & \leq & 4 \,s_k^{\gamma(x_k+2^{-j_k}x)-2} \left( \frac{s_k}{k}\right)^{2-\gamma_\ast(0,1)} \\
& = & 4\,s_k^{\gamma(x_k+2^{-j_k}x)-\gamma_\ast(0,1)} \left( \frac{1}{k}\right)^{2-\gamma_\ast(0,1)}\\
& \leq & \frac{4}{k},
\end{eqnarray*}
for each $x \in B_1$. The last estimate is guaranteed since, for each $k$, 
$$
    \gamma_\ast(0,1) \coloneqq \inf_{y \in B_{1}(0)}\gamma(y) \leq \gamma(x_k+2^{-j_k}x).
$$
Hence,
\begin{equation}\nonumber
    \|\delta_k\|_{L^\infty(B_1)} \leq 4\,\|\delta\|_{L^\infty(B_1)} k^{-1}.
\end{equation}

Next, we apply Theorem \ref{localregthm} for the lower bound 
$$
    \inf_{y \in B_1}\gamma_k(y) = \inf_{y \in B_1} \gamma(x_k+2^{-j_k}y) = \inf_{x \in B_{2^{-j_k}}(x_k)}\gamma(x) = \gamma_\ast(x_k,2^{-j_k}) \geq \gamma_\ast(0,1),
$$
and observe that the sequence $\{\varphi_k\}_k$ is  $C^{1,\frac{\gamma_\ast(0,1)}{2-\gamma_\ast(0,1)}}-$equicontinuous, locally in $B_1$. Therefore, up to a subsequence, $\varphi_k$ converges uniformly to $\varphi_\infty$ locally in $B_1$, as $k\to \infty$.
Taking into account the estimates above, we conclude that $\varphi_\infty$ minimizes the functional 
$$
    v \longmapsto \int_{B_1} \frac{1}{2} \left| Dv \right|^2 \, dx.
$$ 
The proof of this fact follows the same lines as in \cite[Lemma 2.3 and Remark 1]{PT24}. The limit function \(\varphi_\infty\) is harmonic in \(B_1\), with \(\varphi_\infty(0)=0\) but \(\sup_{B_{1/2}} \varphi_\infty = 1\), by \eqref{this better work}. This contradicts the strong maximum principle. 
\end{proof}

Next, we prove a pointwise gradient estimate.

\begin{lemma}\label{pt grad estimate} 
Let $u$ be a local minimizer of the energy-functional \eqref{fifi} in $B_1$. Assume 
$$
    \gamma_\ast (0,1)>0.
$$
There exists a small universal parameter $\tau>0$ and a constant $\overline{C}$, depending only on $\gamma_\ast (0,1)$ and universal parameters, such that if
\begin{equation}\label{1608}
0 \leq u \leq \tau \quad \text{in } \; B_1,
\end{equation}
then
\begin{equation}\label{gradest}
    |Du(x)|^2 \le \overline{C}\, [u(x)]^{\gamma_\ast(0,1)},
\end{equation}
for each $x \in B_{1/2}$.
\end{lemma}

\begin{proof}
The case $x \in \partial\{u>0\} \cap B_{1/2}$ follows from Theorem \ref{localregthm}. Indeed, since solutions are locally $C^{1,\beta}$, for some $\beta>0$, the fact that $u$ attains at each $x\in \partial\{u>0\}$ its minimum value implies that $|Du(x)|=0$. 

We now consider $x \in \{u>0\} \cap B_{1/2}$ and choose
$$
    \tau\coloneqq \frac{1}{C}\left( \frac{1}{8} \right)^{\frac{2}{2-\gamma^\ast(0,1)}},
$$
for $C$ as in Lemma \ref{april.2}. Define
$$
	r \coloneqq (C u(x))^{\frac{2-\gamma_{\ast}(0,1)}{2}},
$$
and observe that for the choice of $\tau$, we have
$$
	r \leq (C \tau)^{\frac{2-\gamma_{\ast}(0,1)}{2}} \leq \frac{1}{4}.
$$
We can then apply Lemma \ref{april.2} to obtain
$$
	\sup_{B_r(x)} u \leq C r^{\frac{2}{2-\gamma_\ast(0,1)}}.
$$
Next, define 
$$
    v(y) \coloneqq  u(x+ry)\,r^{-\frac{2}{2-\gamma_\ast(0,1)}} \quad \text{in } \; B_1,
$$
and observe that it satisfies the uniform bound
$$
    \sup_{B_1} v \leq C.
$$
Additionally, by the scaling properties of Section 2, $v$ is a minimizer of a scaled functional as \eqref{fifi} in $B_1$, and so, by Theorem \ref{localregthm},
$$
    \left| Dv (0) \right| \leq L,
$$
for some $L$, depending only on $\gamma_\ast(0,1)$ and universal parameters. This translates into
\begin{eqnarray*}
\left|Du(x)\right| & \leq & L r^{\frac{\gamma_\ast(0,1)}{2-\gamma_\ast(0,1)}} \\
& = & L \sqrt{C} [u(x)]^{\frac{\gamma_\ast(0,1)}{2}},
\end{eqnarray*}
recalling that $C>1$. Since $0\leq u \leq 1$, the proof follows with $\overline{C} \coloneqq 4L^2C$, which depends only on $\gamma_\ast(0,1)$ and universal parameters.
\end{proof}

\begin{remark}\label{estimate far from FB}
We have proved Lemma \ref{pt grad estimate} under the assumption that \eqref{1608} holds. Observe, however, that the conclusion is trivial otherwise. Indeed, if $u(x) > \tau$, then by Lipschitz regularity we have
$$
    |Du(x)|^2 \leq L^2 = L^2 \left(\frac{\tau}{\tau} \right)^{\gamma_\ast(0,1)} \leq  \frac{L^2}{\tau^{\gamma_\ast(0,1)}} [u(x)]^{\gamma_\ast(0,1)}.
$$
\end{remark}

\section{Weak Dini-continuous exponents and sharp estimates}\label{sct beicodebode}

The local regularity result in Theorem \ref{localregthm} yields a $(1+\alpha)-$growth control for a minimizer $u$ near its free boundary, with 
$$1+\alpha \coloneqq 2/(2-\gamma_\ast(z_0,r)).$$
More precisely, if $z_0$ is a free boundary point then $u(z_0) = |Du (z_0)| = 0$. Consequently, with $r=|y-z_0|$, we have, by continuity,
\begin{eqnarray*}
u(y) & \leq & \sup_{x \in B_{r}(z_0)} \left| u(x)-u(z_0) -Du(z_0)\cdot (x-z_0) \right| \\
& \leq & C r^{1+\alpha}\\
& = & C |y-z_0|^{\frac{2}{2-\gamma_\ast(z_0,r)}}.
\end{eqnarray*}
However, such an estimate is suboptimal, and a key challenge is to understand how the oscillation of $\gamma(x)$ impacts the prospective (point-by-point) $C^{1,\alpha}$ regularity of minimizers along the free boundary. 

In this section, we assume $\gamma$ is continuous at a free boundary point $z_0$, with a modulus of continuity $\omega$ satisfying 
\begin{equation}\label{CMC}
	\omega(1) + \limsup_{t\to 0^+} \ \omega(t) \ln \left( \frac{1}{t} \right ) \le \tilde{C},
\end{equation}
for a constant $\tilde{C}>0$. Such a condition often appears in models involving variable exponent PDEs as a critical (minimal) assumption for the theory; see, for instance, \cite{ABF} for functionals with $p(x)$-growth and \cite{BPRT} for the non-variational theory.

Note that assumption \eqref{CMC} is weaker than the classical notion of Dini continuity. In fact, if \eqref{CMC} is violated, then there exists a decreasing sequence $(t_k)_{k\ge 1}$ with $t_k\downarrow 0$ such that
$$
  \omega(t_k)\,\ln\!\Big(\frac{1}{t_k}\Big)\ \longrightarrow\ \infty \quad \text{as } k\to\infty.
$$
We may also build this sequence such that
$$
  t_{k+1}\ \le\ t_k^{\,2}\qquad\text{for all }k\ge 1,
$$
so the intervals $(t_k,\sqrt{t_k}]$ are pairwise disjoint. Since $\omega$ is nondecreasing, we have
$$
  \int_{t_k}^{\sqrt{t_k}} \frac{\omega(t)}{t}\,dt
  \ \ge\ \omega(t_k)\int_{t_k}^{\sqrt{t_k}} \frac{dt}{t}
  \ =\ \omega(t_k)\,\Big[\ln t\Big]_{t_k}^{\sqrt{t_k}}
  \ =\ \frac{1}{2}\,\omega(t_k)\,\ln\!\Big(\frac{1}{t_k}\Big).
$$
Summing over $k$ and using the disjointness of $(t_k,\sqrt{t_k}]$, we obtain
$$
  \int_{0}^{1} \frac{\omega(t)}{t}\,dt
  \ \ge\ \sum_{k=1}^{\infty} \int_{t_k}^{\sqrt{t_k}} \frac{\omega(t)}{t}\,dt
  \ \ge\ \frac{1}{2}\,\sum_{k=1}^{\infty} \omega(t_k)\,\ln\!\Big(\frac{1}{t_k}\Big)
  \ =\ \infty,
$$
which proves that
$$
  \int_{0}^{1} \frac{\omega(t)}{t}\,dt\ =\ \infty,
$$
and thus, $\gamma$ is not Dini continuous.

\medskip

We are ready to state a sharp pointwise regularity estimate for local minimizers of \eqref{fifi} under \eqref{CMC}. We define the subsets 
$$
	\Omega(u) \coloneqq  \left \{x\in B_{1} \colon u(x) > 0\right \} \quad \text{ and } \quad F(u) \coloneqq  \partial \Omega(u),
$$
corresponding to the non-coincidence set and the free boundary of the problem, respectively. 

\begin{theorem}\label{thm-beicodebode}
	Let $u$ be a local minimizer of \eqref{fifi} in $B_1$ and $z_0 \in F(u) \cap B_{1/2}$. Assume $\gamma$ satisfies \eqref{CMC} at $z_0$. 
Then, there exist universal constants $r_0>0$ and $C^\prime>1$ such that
\begin{equation}\label{thm}
		 u(y)\leq C^\prime \left |y-z_0 \right |^{\frac{2}{2-\gamma(z_0)}},
	\end{equation}
for all $y \in B_{r_0}(z_0)$.
\end{theorem}

\begin{proof}
Since \eqref{CMC} is in force, let $r_0 \ll 1$ be such that, for $r< r_0$, 
\begin{equation}\label{vergalhão}
\omega(r) \ln \left( \frac{1}{r} \right ) \le 2 \left[ \tilde{C}-\omega(1)\right] =: C^\ast.
\end{equation}
Fix $y \in B_{r_0}(z_0)$ and let 
$$r\coloneqq |y-z_0| < r_0.$$
Apply Theorem \ref{localregthm} to $u$ over $B_{r}(z_0)$, to get
$$\sup_{x\in B_r(z_0)} u(x) \leq C\, r^{\frac{2}{2-\gamma_\ast(z_0,r)}}.$$
In particular, by continuity, it follows that
\begin{equation}\label{sct3-Eq2}
	u(y) \leq C \, r^{\frac{2}{2-\gamma_\ast(z_0,r)}}.
\end{equation}
In view of \eqref{CMC}, we can estimate
$$\gamma(z_0) - \gamma_\ast(z_0,r) \le \omega(r),$$ 
and, since the function $g\colon [0,1] \to [1,2]$ given by
$$
	g(t) \coloneqq  \frac{2}{2-t}
$$
satisfies $g'(t) \le 2$, for all $t\in [0,1]$, we have
\begin{eqnarray*}
		 g\left (\gamma(z_0) \right ) - g\left (\gamma_\ast(z_0,r) \right )  &\le& 2 \left (\gamma(z_0) - \gamma_\ast(z_0,r) \right ) \\
		&\le & 2\, \omega(r ).
\end{eqnarray*}
Combining \eqref{sct3-Eq2} with this inequality, and taking \eqref{vergalhão} into account, we reach 
\begin{eqnarray*}
		\displaystyle u(y) &\leq& \displaystyle C  \, r^{-\left [g\left (\gamma(z_0) \right ) - g\left (\gamma_\ast(z_0,r) \right ) \right ]}\, r^{\frac{2}{2-\gamma(z_0)}} \\
		&\le & C\, r^{-2\omega(r)} \, r^{\frac{2}{2-\gamma(z_0)}} \\
		&\le & C\, e^{2C^\ast} \, r^{\frac{2}{2-\gamma(z_0)}} \\
		&=& C^\prime \, |y-z_0|^{\frac{2}{2-\gamma(z_0)}},
\end{eqnarray*}
as desired.  
\end{proof}

We also obtain a sharp strong non-degeneracy result.

\begin{theorem}\label{sharp nondeg} Let $u$ be a local minimizer of \eqref{fifi} in $B_1$ and $z_0 \in F(u) \cap B_{1/2}$. Assume \eqref{bound delta} and that \eqref{CMC} is in force at $z_0$. Then, there exist universal constants $r_0>0$ and $c^\ast> 0$ such that
$$
	\sup_{\partial B_r(z_0)} u \geq c^\ast \, r^{\frac{2}{2-\gamma(z_0)}},
$$
for every $0<r < r_0$. 
\end{theorem}

\begin{proof}
As before, let $r_0 \ll 1$ be such that \eqref{vergalhão} holds and fix $r<r_0$. From Theorem \ref{nondeg}, we know
$$\sup_{\partial B_r(z_0)} u \geq c \, r^{\frac{2}{2-\gamma^\ast(z_0,r)}},$$
with $c>0$ depending only on $n$, $\delta_0$ and $\gamma_\ast (0,1)$.

Now, observe that 
$$\frac{2}{2-\gamma^\ast(z_0,r)} = \frac{2}{2-\gamma(z_0)} + \frac{2}{2-\gamma^\ast(z_0,r)} - \frac{2}{2-\gamma(z_0)}$$
and
\begin{eqnarray*}
\displaystyle \frac{2}{2-\gamma^\ast(z_0,r)} - \frac{2}{2-\gamma(z_0)} & = &  \displaystyle \frac{2 ( \gamma^\ast(z_0,r) -\gamma(z_0) )}{( 2- \gamma^\ast(z_0,r)) \, ( 2-\gamma(z_0) )} \\
& \leq & \displaystyle  2 ( \gamma^\ast(z_0,r) -\gamma(z_0) )\\
& \leq & \displaystyle  2 \omega(r). 
\end{eqnarray*}
Thus,
\begin{eqnarray*}
\displaystyle r^{\frac{2}{2-\gamma^\ast(z_0,r)}} & \geq&   \displaystyle r^{2 \omega(r)} r^{\frac{2}{2-\gamma(z_0)}} \\
& = &  \displaystyle e^{2 \omega(r) \ln r} \, r^{\frac{2}{2-\gamma(z_0)}}\\
& \geq &  \displaystyle e^{-2 C^{\ast}} \, r^{\frac{2}{2-\gamma(z_0)}},
\end{eqnarray*}
due to \eqref{vergalhão}, and the result follows with $c^\ast \coloneqq  c \, e^{-2 C^{\ast}}$.
\end{proof}	

With sharp regularity and non-degeneracy estimates at hand, we can now prove the positive density of the non-coincidence set and the porosity of the free boundary. Recall that a set $E \subset \mathbb{R}^n$ is said to be porous if there exists a constant $\kappa \in (0,1)$ and $r_0>0$ such that, for every $r \leq r_0$ and every $x \in E$, there is $y \in \mathbb{R}^n$ such that
$$
    B_{\kappa r}(y) \subseteq B_r(x) \setminus E.
$$

\begin{theorem}\label{pos_density} Let $u$ be a local minimizer of \eqref{fifi} in $B_1$ and $z_0 \in F(u) \cap B_{1/2}$. Assume \eqref{bound delta} and that \eqref{CMC} is in force at $z_0$. There exists a constant $\mu_0 > 0$, depending on $n$, $\delta_0$, $\gamma_\ast (0,1)$ and the constant from \eqref{CMC}, such that
$$
	\frac{\left |B_r(z_0) \cap \Omega(u) \right |}{\left |B_r(z_0) \right |} \ge \mu_0,
$$
for every $0<r<r_0.$ In particular, $F(u)$ is porous and there exists an $\epsilon>0$ such that $\mathcal{H}^{n-\epsilon}(F(u) \cap B_{1/2}) = 0$.
\end{theorem}

\begin{proof}
Fix $r < r_0$, with $r_0$ as in Theorem \ref{thm-beicodebode}. It follows from the non-degeneracy (Theorem \ref{sharp nondeg}) that there exists $y \in \partial B_r(z_0)$ such that
$$u (y) \geq c^\ast r^{\frac{2}{2-\gamma(z_0)}}.$$
Now, let $z \in F(u)$ be such that 
$$
    \left| z-y \right| = \dist (y, F(u)) \eqqcolon d.
$$
Then, we have
$$
    c^\ast r^{\frac{2}{2-\gamma(z_0)}} \leq u(y) \leq \sup_{B_d(z)} u \leq C d^{\frac{2}{2-\gamma(z)}}.
$$
Furthermore, observe that
$$
    |z-z_0| \leq |z-y| + |y-z_0| \leq d+r,
$$
and so, since $d \leq r$, we have $|z-z_0| \leq 2r$. Therefore, one can proceed as in Theorem \ref{thm-beicodebode} to obtain
$$
    c^\ast r^{\frac{2}{2-\gamma(z_0)}} \leq u(y) \leq Cd^{\frac{2}{2-\gamma(z_0)}}.
$$
This implies that 
$$
    r  \leq  \left(\frac{C}{c^\ast} \right)^{\frac{2-\gamma(z_0)}{2}}  d \leq  \max \left\{ 1, \frac{C}{c^\ast} \right\} d.
$$
So for $\kappa = \min \left\{ 1, c^\ast / C \right\}$, we have
$$
    B_{\kappa r}(y) \subset B_d(y) \subset \Omega(u).
$$

Since also $B_{\kappa r}(y) \subset B_{2r}(z_0)$, we conclude
$$
    |B_{2r}(z_0)\cap \Omega(u)| \geq \left(\frac{\kappa}{2} \right)^n \alpha(n) (2r)^n,
$$
where $\alpha(n)$ is the volume of the unit ball in $\R^n$, and the result follows with $\mu_0 = \left(\frac{\kappa}{2} \right)^n$. 

We have shown that for any $z \in F(u)$, there exists a point $y_z \in \Omega(u)$ and a universal constant $\kappa > 0$ such that 
$B_{\kappa r}(y_z) \subset \Omega(u) \cap B_r(z)$. In particular,
$$
    B_{\kappa r}(y_z) \subseteq B_r(z) \setminus F(u).
$$
The last implication follows from classical results in \cite{Z}.
\end{proof}

Assuming $\gamma$ is continuous, with modulus of continuity satisfying \eqref{CMC}, we can get an improved version of \eqref{april.2} with optimal exponents.

\begin{lemma}\label{april} 
Let $u$ be a local minimizer of the energy-functional \eqref{fifi} in $B_1$. Assume that
$$
    \gamma_\ast (0,1)>0,
$$
and $\gamma$ is continuous with modulus of continuity satisfying \eqref{CMC}. There exists a constant $C>4$, depending only on $\gamma_\ast (0,1)$ and universal parameters, such that, if 
\begin{equation*}
    u(x) \le \frac{1}{C} \, r^{\frac{2}{2-\gamma(x)}}, 
\end{equation*}
for $x\in B_{1/2}$ and $r \leq 1/4$, then
\begin{equation}\label{conclusion flat lemma}
    \sup_{B_r(x)} u \leq C r^{\frac{2}{2-\gamma(x)}}.
\end{equation}  
\end{lemma} 

\begin{proof}
The proof follows similar steps as in Lemma \ref{april.2}, and we only highlight the main steps. Fix $r<1/4$ and consider $j_r \in \mathbb{N}$ to be such that
$$
    2^{-(j_r+1)} \leq r < 2^{-j_r}.
$$
For $j \in \{1, 2, \cdots, j_r \}$, define
$$
    S_j(x,u) \coloneqq \sup_{B_{2^{-j}}(x)}u \qquad \text{and} \qquad a_j \coloneqq 2^{j \beta(x)}S_j(x,u),
$$
where
$$
    \beta(x) = \frac{2}{2-\gamma(x)},
$$
and for $j=j_r+1$ we define
$$
    a_{j_r + 1} \coloneqq r^{\frac{-2}{2-\gamma(x)}}\sup_{B_r(x)}u.
$$
Then, to obtain \eqref{conclusion flat lemma}, it is enough to prove 
\begin{equation}\label{recursive arg shaggy 2}
    a_{j+1} \leq \max\{C,a_j\}, \quad \forall j \in \{1, 2, \cdots, j_r\}.
\end{equation}
Let us now suppose, seeking a contradiction, that \eqref{recursive arg shaggy 2} fails. Then, for each integer $k>0$, there exist a minimizer $u_k$ of \eqref{fifi} in $B_1$, $x_k \in B_{1/2}$ and $ 0< r_k < 1/4$, such that 
$$
    u_k(x_k) \le \frac{1}{k} \, r_k^{\beta(x_k)},
$$
but
\begin{equation}\label{equiv inequality 2}
    a_{j_k + 1} > \max\{k,a_{j_k}\}, \qquad \text{for some} \ j_k \in \{1,2,\cdots, j_{r_k}\}.
\end{equation}
In the sequel, define
$$
    \varphi_k(x)\coloneqq \frac{u_k(x_k+2^{-j_k}x)}{S_{j_k + 1}(x_k,u_k)} \quad \text{in} \; B_1.
$$
For this function, there holds
\begin{equation}\label{this better work 2}
    \sup_{B_1} \varphi_k \leq 4, \qquad \sup_{B_{1/2}} \varphi_k = 1, \qquad \text{and} \quad \varphi_k(0) = O(k^{-2}).
\end{equation}
Indeed, from \eqref{equiv inequality 2}, we obtain
\begin{align*}
    \sup_{B_1} \varphi_k  & = \frac{S_{j_k}(x_k,u_k)}{S_{j_k + 1}(x_k,u_k)}\\ 
    & < \frac{2^{(j_k+1)\beta(x_k)}}{2^{j_k\beta(x_k)}} \leq 4 .
\end{align*}
From scaling, it directly follows that $\sup_{B_{1/2}} \varphi_k = 1$, and finally,
$$
    \varphi_k(0) \leq \frac{1}{k^2}\,\frac{r_k^{\beta(x_k)}}{2^{(j_k+1)\beta(x_k)}} \leq \frac{1}{k^2}.
$$
In addition, note that $\varphi_k$ minimizes
$$
    v \longmapsto \int_{B_1} \frac{1}{2} \left| Dv \right|^2 + \delta_k(x) v^{\gamma_k(x)}\, dx,
$$
for
$$
    \delta_k(x)\coloneqq  \delta(x_k+2^{-j_k}x)\frac{2^{-2j_k}}{s_k^{2-\gamma(x_k+2^{-j_k}x)}} \quad \text{and} \quad \gamma_k(x)\coloneqq \gamma(x_k+2^{-j_k}x),
$$ 
where
$$
    s_k \coloneqq S_{j_k + 1}(x_k,u_k).
$$
From \eqref{equiv inequality 2}, we obtain
\begin{eqnarray*}
s_k^{\gamma(x_k+2^{-j_k}x)-2}2^{-2j_k} & \leq & 4\, s_k^{\gamma(x_k+2^{-j_k}x)-2} \left( \frac{s_k}{k}\right)^{2-\gamma(x_k)} \\
& = & 4\, s_k^{\gamma(x_k+2^{-j_k}x)-\gamma(x_k)} \left( \frac{1}{k}\right)^{2-\gamma(x_k)}\\
& \leq & \frac{C}{k},
\end{eqnarray*}
for each $x \in B_1$ and for some universal constant $C$. The last estimate is guaranteed since $s_k = O(2^{-j_k \beta(x_k)})$, and so we can uniformly bound the term
$$
    s_k^{\gamma(x_k+2^{-j_k}x)-\gamma(x_k)}.
$$
Hence,
\begin{equation}\nonumber
    \|\delta_k\|_{L^\infty(B_1)} \leq C \, \|\delta\|_{L^\infty(B_1)} k^{-1}.
\end{equation}
We can then apply Theorem \ref{localregthm} to get a contradiction by passing to the limit.
\end{proof}

With this result, we are able to establish an optimized version of Lemma \ref{pt grad estimate}, assuming that $\gamma(x)$ satisfies condition \eqref{CMC}. 

\begin{lemma}\label{opt pt grad estimate} 
Let $u$ be a local minimizer of the energy-functional \eqref{fifi} in $B_1$. Assume \eqref{bound delta} is in force, and $\gamma$ is continuous with modulus of continuity satisfying \eqref{CMC}. There exists a constant $C$, depending on $\gamma_\ast (0,1)$ and universal parameters, such that
\begin{equation}\nonumber
|Du(x)|^2 \le C\, [u(x)]^{\gamma(x)},
\end{equation}
for each $x \in B_{1/2}$.
\end{lemma}

\begin{proof}
The proof is essentially the same as the proof of Lemma \ref{pt grad estimate}, except for the steps we highlight below. By Remark \ref{estimate far from FB}, it is enough to prove the result at points such that $0 \leq u(x) \leq \tau$. 
Choose $r$ so that
$$
  r^{\frac{2}{2-\gamma(x)}} = C\,u(x),
$$
which can be made small (depending on $\tau$). By Lemma \ref{april}, the rescaled
$$
  v(y) \coloneqq u(x+ry)\,r^{-\frac{2}{2-\gamma(x)}}
$$
is uniformly bounded in $B_1$. The scaling from Section \ref{prelim-sect} yields
$$
  \|\tilde{\delta}\|_{L^\infty(B_1)}
  \ \le\ r^{\frac{2}{2-\gamma(x)}\,(\gamma_\ast(x,r)-2)}\,r^2\,\|\delta\|_{L^\infty(B_1)}
  \ =\ r^{2(\gamma_\ast(x,r)-\gamma(x))}\,\|\delta\|_{L^\infty(B_1)},
$$
which is uniformly bounded thanks to \eqref{CMC}. Lipschitz bounds for $v$ follow, and the claim is proved. 
\end{proof}

\begin{example} We conclude this section with an insightful observation leading to a class of intriguing free boundary problems. Initially, it is worth noting that the proof of the existence of a minimizer can be readily adapted for more general energy-functionals of the form
\begin{equation}\label{funct ex}
    J(v) \coloneqq \int \frac{1}{2} |Dv|^2 + \delta(x) (v^+)^{\gamma(x, v(x))} \, dx,
\end{equation}
provided $\gamma \colon \Omega \times \mathbb{R} \to \mathbb{R}$ is a Carath\'eodory function, which is bounded from above away from $1$ and bounded from below away from $0$. We further emphasize that our local $C^{1,\alpha}$ regularity result, Theorem \ref{localregthm}, also applies to this class of functionals. 

To illustrate the applicability of these results, let us consider the following toy model, where the varying singularity $\gamma(x,v)$ is given only globally measurable and bounded, such that $\gamma(x,v) \geq 1/6$, and
\begin{equation}
\gamma(x,v)=
\dfrac{1}{2}-\dfrac{3}{(\ln(\min(v(x),e^{-3})))^{2}}.
\end{equation}
The function $\gamma$ is Dini continuous, but not better, at the region
$$\{u = 0\} \supset F(u),$$
for any minimizer $u$ of the corresponding functional $J$ in \eqref{funct ex}. Since
$$
\gamma_\ast(0,1)= \dfrac 16,
$$
the local regularity estimate obtained in Theorem \ref{localregthm}, gives that minimizers are locally of class $C^{12/11}$. In contrast, observe that 
$$
\gamma \equiv \dfrac{1}{2} \quad \text{at } \; F(u),
$$
and so, Theorem \ref{thm-beicodebode} asserts that local minimizers are {\it precisely} of class $C^{4/3}$ at free boundary points. A wide range of meaningful examples can be constructed out of functions obtained in \cite[{\it Section} 2]{APPT}.

Applying a similar reasoning, we can provide examples of energy-function\-als for which minimizers are locally of class $C^{1,\epsilon}$, for $0<\epsilon \ll 1$, whereas along the free boundary, they are $C^{1,1-\epsilon}-$regular. We anticipate revisiting the analysis of such models in future investigations. 
\end{example}

\section{Hausdorff measure estimates}\label{sct Hausdorff}

In this section, we prove Hausdorff measure estimates for the free boundary under the stronger regularity assumptions on the data
\begin{equation}\label{nunca derivei tanto}
    \delta(x) \in W^{2,\infty}(B_1) \quad \text{and} \quad \gamma(x) \in W^{2,\infty}(B_1).
\end{equation}
Differentiability of the free boundary will be obtained in Section \ref{sct FB reg}, assuming only $\delta, \gamma \in W^{1,q}(B_1)$, for some $q>n$.

Furthermore, we shall also assume
\begin{equation}\label{gamma cima menor q 1}
        \gamma^\ast(0,1) \coloneqq \gamma^\ast(B_1(0)) < 1.
\end{equation}

We will need a few preliminary results, as in \cite{AP}. We begin with a slightly different pointwise gradient estimate compared to Lemma \ref{opt pt grad estimate}.

\begin{lemma}\label{slightly dif pt grad est}
Let $u$ be a local minimizer of the energy-functional \eqref{fifi} in $B_1$. Assume \eqref{bound delta}, \eqref{CMC}, \eqref{gamma cima menor q 1}, and \eqref{nunca derivei tanto} are in force and let $x_0 \in F(u) \cap B_{1/2}$. There exists a constant $c_1$, depending only on $n$, $\delta_0$, $\gamma_\ast (0,1)$, $\|D\delta\|_\infty$, $\|D^2\delta\|_\infty$, $\|D\gamma\|_\infty$ and $\|D^2\gamma\|_\infty$, such that
$$
    |Du (x)|^2 \leq 2\delta(x)\,[u(x)]^{\gamma(x)} + c_1u(x),
$$
for each $x \in B_{1/8} (x_0)$.
\end{lemma}

\begin{proof}
Let $K_0>0$, $\tau = 1/8$ and consider $\zeta \colon [0,3\tau] \mapsto \mathbb{R}$, defined by 
\[
\zeta(t)=
\begin{cases}
  0, & t\in[0,\tau],\\
  K_0\,(t-\tau)^3, & t\in[\tau,3\tau].
\end{cases}
\]
For $K_1>0$ a large constant to be chosen later, define
$$
    w(y) \coloneqq |Du(y)|^2 - 2\delta(y) [u(y)]^{\gamma(y)} - K_1 u(y) - \zeta(|y-x_0|) [u(y)]^{\gamma(y)},
$$
for $y \in \Omega(u) \cap B_{3\tau}(x_0)$. By Lemma \ref{opt pt grad estimate}, we can suitably choose $K_0>0$ so that $w \leq 0$ on $\partial B_{3\tau}(x_0)$, and so $w \leq 0$ on $\partial (\Omega(u) \cap B_{3\tau}(x_0))$, since $w \equiv 0$ on $\partial \Omega(u)$. Indeed, from that lemma, there is a constant $C$ such that
$$
    |Du(x)|^2 \leq C[u(x)]^{\gamma(x)}, \quad \text{for} \quad x \in B_{1/2}.
$$
For $y \in \partial B_{3\tau}(x_0)$, we have
\begin{align*}
    w(y) &\leq |Du(y)|^2 - \zeta(|y-x_0|)[u(y)]^{\gamma(y)}\\
         & = |Du(y)|^2 - 8\tau^3{K_0}[u(y)]^{\gamma(y)}\\
         & \leq |Du(y)|^2 - C[u(y)]^{\gamma(y)},
\end{align*}
as long as $K_0$ is chosen such that $8 \tau^3 K_0 \geq C$. We will show that $w \leq 0$ in $\Omega(u) \cap B_{3\tau}(x_0)$. To do so, we assume, to the contrary, that $w$ attains a positive maximum at $p \in \Omega(u) \cap B_{3\tau}(x_0)$. Since $w$ is smooth within $\Omega(u)$ and $p$ is a point of maximum for $w$, we have $\Delta w(p) \leq 0$. To reach a contradiction, we will show that if $K_1>0$ is chosen large enough, then $\Delta w(p) > 0$.

We will omit the point $p$ whenever possible to ease the notation. We also rotate the coordinate system so that $e_1$ is in the direction of $Du(p)$. We then have
\begin{eqnarray*}
    0 & = & \partial_1 w(p)  \\
      & = &  2Du \cdot D\left(\partial_1u\right) - 2\left(\partial_1\delta\right) u^{\gamma} - 2\delta \left(\gamma u^{\gamma-1}\left(\partial_1u\right) + \left(\partial_1\gamma\right) u^\gamma \ln(u) \right) \\
      &  & - K_1\left(\partial_1u\right) - \left(\partial_1\zeta\right) u^\gamma - \zeta\left(\gamma u^{\gamma-1}\left(\partial_1u\right) + \left(\partial_1\gamma\right)\,u^\gamma \ln(u) \right)\\
      & = &   \left(\partial_1 u\right)\left[2\left(\partial_{11}u\right) - \frac{u^\gamma}{\left(\partial_1 u\right)}(2\left(\partial_1 \delta\right) + \left(\partial_1 \zeta\right)) - \gamma(2\delta + \zeta)u^{\gamma - 1} - K_1\right]\\
      & & -  \left(\partial_1 u\right) \left[ \frac{u^\gamma}{\left(\partial_1 u\right)}\left(\partial_1 \gamma\right) \ln(u)(2\delta  + \zeta)\right].
\end{eqnarray*}
Since $\partial_1u(p) >0$, we obtain
\begin{eqnarray*}
    2\left(\partial_{11}u\right) & = &\frac{u^\gamma}{\left(\partial_1 u\right)}(2\left(\partial_1\delta\right) + \left(\partial_1 \zeta\right)) + \gamma(2\delta + \zeta)u^{\gamma - 1} + K_1\\
    & & + \frac{u^\gamma}{\left(\partial_1 u\right)}\left(\partial_1 \gamma\right) \ln(u)(2\delta  + \zeta).
\end{eqnarray*}
Moreover, since $w(p)>0$, we have
$$
	\left(\partial_1 u(p)\right)^2 > 2\delta(p) [u(p)]^{\gamma(p)} + K_1 u(p) + \zeta(|p-x_0|) [u(p)]^{\gamma(p)} \geq 2\delta(p) [u(p)]^{\gamma(p)},
$$
from which follows that $\partial_1 u(p) > \sqrt{2\delta(p)}u(p)^{\frac{\gamma(p)}{2}}$. As a consequence,
$$
    \frac{u^\gamma}{\left(\partial_1 u\right)} \leq \frac{u^\frac{\gamma}{2}}{\sqrt{2\delta}} \leq \frac{1}{\sqrt{2\delta_0}} u^\frac{\gamma}{2}.
$$
This implies that
$$
    2\left(\partial_{11}u\right) \geq 2\delta \gamma u^{\gamma-1} + K_1 + \zeta \gamma u^{\gamma-1} - C_1 u^{\frac{\gamma}{2}} - C_2 u^{\frac{\gamma}{2}}|\ln(u)|,
$$
for constants $C_1 = C_1(\delta_0,\|D\delta\|_\infty,K_0)$ and $C_2 = C_2(\delta_0,\|D\gamma\|_\infty, K_0, \|\delta\|_\infty)$. For a small $\eta^\ast>0$ so that $\gamma/2 - \eta^\ast >0$ and a larger constant $C_3$, we then have
\begin{eqnarray*}
	2\left(\partial_{11}u\right) &\geq & 2\delta \gamma u^{\gamma-1} + K_1 + \zeta \gamma u^{\gamma-1} - C_3 u^{\frac{\gamma}{2}-\eta^\ast}\\
			& = & 2\delta \gamma u^{\gamma-1} + \eta K_1 + \zeta \gamma u^{\gamma-1} + (1-\eta)K_1 - C_3 u^{\frac{\gamma}{2}-\eta^\ast},
\end{eqnarray*}
for $\eta \coloneqq 3/4$, where we used that $u^{\eta^\ast}|\ln(u)|$ is bounded for $u \leq 1$. This fact will be used throughout this section to handle the $\log$-terms. For large $K_1$, it follows that $(1-\eta)K_1 - C_3 u^{\frac{\gamma}{2}-\eta^\ast} \geq 0$, and so
$$
	2\left(\partial_{11}u\right) \geq 2\delta \gamma u^{\gamma-1} +  \eta K_1 + \zeta \gamma u^{\gamma-1}.
$$
Squaring both sides gives
\begin{align*}
4\bigl(\partial_{11}u\bigr)^2 
&\geq \bigl(2\delta \gamma u^{\gamma-1} + \eta K_1 + \zeta \gamma u^{\gamma-1}\bigr)^2 \\[0.5em]
&\geq \bigl(2\delta \gamma u^{\gamma-1}\bigr)^2
   + 2\bigl(2\delta \gamma u^{\gamma-1}\bigr)(\eta K_1)
   + 2\bigl(2\delta \gamma u^{\gamma-1}\bigr)(\zeta \gamma u^{\gamma-1}),
\end{align*}
and so
\begin{equation}\label{ineq for 2 derivative squared}
    \left(\partial_{11}u\right)^2 \geq \left(\delta \gamma u^{\gamma-1}\right)^2 + \delta \gamma \eta K_1 u^{\gamma-1} + \delta \zeta \left(\gamma u^{\gamma-1}\right)^2.
\end{equation}
Now, we calculate $\Delta w$ at the point $p$. By direct computations, we obtain
\begin{eqnarray*}
            \Delta w & = &  2 \sum_{k,j} (\partial_{k,j} u)^2 + 2Du \cdot D(\Delta u) - 2u^\gamma \Delta \delta - 4 D\delta \cdot D(u^\gamma)  \\
             & & - 2\delta\, \Delta(u^\gamma) - K_1\Delta u - u^\gamma \Delta \zeta - 2D\zeta \cdot D(u^\gamma) - \zeta \Delta (u^\gamma).
\end{eqnarray*}
Moreover,
\begin{eqnarray*}
            D(u^\gamma) & = & u^\gamma \ln(u) D\gamma + \gamma u^{\gamma-1}Du,\\[0.2cm]
            \Delta(u^\gamma) & = & u^\gamma \ln(u) \Delta \gamma + u^\gamma (\ln(u))^2 |D\gamma|^2 + 2\gamma u^{\gamma-1} \ln(u) D\gamma \cdot Du \\
             &  & +\, 2u^{\gamma-1} D\gamma \cdot Du + \gamma(\gamma -1)u^{\gamma-2}|Du|^2 + \gamma u^{\gamma-1}\Delta u.
\end{eqnarray*}
To estimate those terms, we take into account that each factor that contains $Du$ can be further estimated by using Lemma \ref{opt pt grad estimate}. Also, each term that contains the derivatives of $\gamma$, can be controlled by $u^{\gamma-1}$, which is possible since $\gamma - 1 < 0$. This gives
$$
    |D(u^\gamma)| \leq C_4 u^{\frac{3\gamma}{2}-1} \leq C_4u^{\gamma - 1},
$$
for $C_4 = C_4(n, \gamma_\ast(0,1), \|\delta\|_\infty, \|D\gamma\|_\infty)$. Furthermore,
$$
    \Delta (u^\gamma) \leq C_5 u^{\gamma-1} + \delta \gamma^2 u^{2\gamma-2}\left[ \frac{(\gamma-1)}{\gamma \delta} \frac{|Du|^2}{u^\gamma} +1 \right],
$$
for a constant $C_5 = C_5(n, \gamma_\ast (0,1), \gamma^\ast(0,1), \|\delta\|_\infty,\|D\gamma\|_\infty, \|D^2 \gamma\|_\infty)$. One can now further estimate $\Delta w$ from below to obtain
\begin{eqnarray*}
     \Delta w & \geq & \displaystyle 2  (\partial_{11} u)^2 - C_6u^{\gamma-1} + 2\delta \gamma (\gamma-1)u^{\gamma-2}|Du|^2\\
            & &\displaystyle  - 2\delta^2 \gamma^2 u^{2\gamma-2} \left[ \frac{(\gamma-1)}{\gamma \delta} \frac{|Du|^2}{u^\gamma} +1 \right] -K_1 \delta \gamma u^{\gamma-1}\\
            & & \displaystyle -\delta \zeta \gamma^2 u^{2\gamma-2} \left[ \frac{(\gamma-1)}{\gamma \delta} \frac{|Du|^2}{u^\gamma} +1 \right]\\[0.4cm]
            & \geq & \displaystyle 2  (\partial_{11} u)^2 - C_6u^{\gamma-1} - K_1 \delta \gamma u^{\gamma-1}- 2\delta^2 \gamma^2 u^{2\gamma-2}  -\delta \zeta \gamma^2 u^{2\gamma-2}, 
\end{eqnarray*}
where we used that $\gamma - 1 < 0$ to disregard the first term in the last bracket, and have used Lemma \ref{opt pt grad estimate} again to estimate
\begin{align*}
    Du \cdot D(\Delta u) & = Du \cdot D(\delta \gamma u^{\gamma-1})\\
                         & =  \delta \gamma ((\gamma-1)u^{\gamma-2}|Du|^2 + u^{\gamma-1}\ln(u)Du \cdot D\gamma)\\
                         & \quad + u^{\gamma-1} Du \cdot (\delta D\gamma + \gamma D\delta)\\
                         & \geq \delta \gamma (\gamma-1)u^{\gamma-2}|Du|^2 - O\left(u^{\gamma-1}\right).
\end{align*}       
We now use \eqref{ineq for 2 derivative squared} to estimate further the second derivative from below, which gives
\begin{eqnarray*}
            \Delta w & \geq & 2\left(\delta \gamma u^{\gamma-1}\right)^2 + 2\delta \gamma \eta K_1 u^{\gamma-1} + 2\delta \zeta \left(\gamma u^{\gamma-1}\right)^2\\
            & & - C_6u^{\gamma-1} - K_1 \delta \gamma u^{\gamma-1}- 2\delta^2 \gamma^2 u^{2\gamma-2}  -\delta \zeta \gamma^2 u^{2\gamma-2}\\           
            & \geq &  2\delta \gamma \eta K_1 u^{\gamma-1} - C_6u^{\gamma-1} - K_1\delta \gamma u^{\gamma-1}\\
            & = & u^{\gamma-1} \left[2\delta\gamma \eta K_1 - C_6 - K_1\delta \gamma   \right].
\end{eqnarray*}
Now, recalling that $\eta = 3/4$, we can choose $K_1$ so large that the above expression is positive. This leads to a contradiction, as discussed before. Since $\zeta$ vanishes on $B_\tau(x_0)$, the result is proved.
\end{proof}

The second preliminary result concerns the integrability of a negative power of the minimizer.

\begin{lemma}\label{integrability of negative power of u}
    Let $u$ be a local minimizer of the energy-functional \eqref{fifi} in $B_1$. Assume \eqref{bound delta}, \eqref{CMC}, \eqref{nunca derivei tanto}, and \eqref{gamma cima menor q 1} are in force. If $0 \in F(u)$, then
    $$
        u(x)^{-\frac{\gamma(x)}{2}} \in L^1 (\Omega(u) \cap B_{1/2}).
    $$
\end{lemma}
    
\begin{proof}
Observe that it is enough to show that
\begin{equation}\label{local_integrability}
    u(x)^{-\frac{\gamma(x)}{2}} \in L^1(\Omega(u) \cap B_\tau(z)), 
\end{equation}
for some small $\tau>0$ and every $z \in F(u)$. Indeed, once this is proved, we can cover $F(u) \cap B_{1/2}$ with finitely many balls with radius $\tau>0$, say $\{B_\tau(z_i) \}$. Then,
$$
    \int\limits_{\Omega(u) \cap \left(\cup B_\tau(z_i)\right)}u^{-\frac{\gamma(x)}{2}}\, dx \leq \sum_{i} \int\limits_{\Omega(u) \cap B_\tau(z_i)}u^{-\frac{\gamma(x)}{2}}\, dx \leq C.
$$
Also, by continuity of $u$ and nondegeneracy (Theorem \ref{nondeg}), we have
$$
    u \geq c \quad \text{in} \quad \left(\Omega(u) \cap B_{1/2} \right) \setminus \bigcup_{i} B_\tau(z_i),
$$
from which the statement in the lemma follows. 

To prove \eqref{local_integrability}, we follow closely the argument in \cite[Lemma 2.5]{P2}. Set
$$
	w \coloneqq u^{2-\frac{3}{2}\gamma(x)}.
$$  First, take $\rho \in C^\infty(\mathbb{R}^+)$, satisfying $\rho' \geq 0$, $\rho \equiv 0$ in $[0,1/2]$ and $\rho(t) = t$ in $[1,\infty)$. For $\delta>0$, let $\rho_\delta(t) \coloneqq \delta \rho(\delta^{-1}t)$. If $\delta < \epsilon$, then the quantity
\begin{equation}\label{penalization argument}
    A \coloneqq \frac{1}{\epsilon} \int_{\{0 \leq u < \epsilon \} \cap B_\tau(z)} \left(Dw \cdot Du\right) \, \rho_\delta'(u)\,dx
\end{equation}
can be written as
$$
    \frac{1}{\epsilon} \int_{ B_\tau(z)} Dw \cdot D\left(\rho_\delta(\min(u,\epsilon))\right)\, dx,
$$
where $z \in F(u)\cap B_{1/2}$. Integrating this quantity by parts, we obtain
$$
    A = - \frac{1}{\epsilon} \int_{B_\tau(z)} \rho_\delta(\min(u,\epsilon)) \Delta w\,dx + \int_{\partial B_\tau(z)} \frac{\rho_\delta(\min(u,\epsilon) )}{\epsilon} \partial_{\nu} w\, d\mathcal{H}^{n-1}.
$$
Pick $\delta = \epsilon/2$. Taking into account that $\rho_\delta(u) = 0$ in the set $\{ 0 < u \leq \epsilon/4 \}$, we have
\begin{eqnarray*}
        A & = & \displaystyle -\frac{1}{2}\int_{\{ \epsilon/4 < u \leq \epsilon\}\cap B_\tau(z)} \rho\left(\frac{2}{\epsilon} u\right) \Delta w\,dx  - \int_{\{ \epsilon < u \}\cap B_\tau(z)} \Delta w\,dx \\
        &  & +\displaystyle \int_{\partial B_\tau(z)} \frac{\rho_\delta(\min(u,\epsilon) )}{\epsilon} \partial_{\nu} w\, d\mathcal{H}^{n-1},
\end{eqnarray*}
where we used that 
$$
	\frac{1}{\epsilon}\rho_\delta(\min(u,\epsilon)) = 1, \text{ in the set }\, \{u>\epsilon\}.	
$$
By Lemma \ref{opt pt grad estimate}, we have
\begin{eqnarray*}
        |Dw(x)| & \leq & \displaystyle 2|D\gamma(x)|u(x)^{2-\frac{3}{2}\gamma(x)} |\ln(u(x))|\\
        &  & + \displaystyle  \left(2-\frac{3}{2}\gamma(x)\right)u(x)^{1-\frac{3}{2}\gamma(x)}|Du(x)|  \\
         & \leq & \displaystyle C\left(|D\gamma(x)|+1\right),
\end{eqnarray*}
for some universal constant $C>0$, and so
\begin{equation}\label{pen arg II}
    A \leq C\tau^{n-1} - \frac{1}{2}\int_{\{ \epsilon/4 < u \leq \epsilon\}\cap B_\tau(z)} \rho\left(\frac{2}{\epsilon} u\right) \Delta w\,dx  - \int_{\{ \epsilon < u \}\cap B_\tau(z)} \Delta w\,dx.
\end{equation}
By direct computations, it follows that
\begin{eqnarray*}
\Delta w(x) & = & a(x) + \left(2-\frac{3}{2}\gamma(x)\right) \left( \left(1-\frac{3}{2}\gamma(x)\right)u(x)^{-\frac{3}{2}\gamma(x)}|Du(x)|^2 \right.\\
& & \qquad \qquad \qquad \qquad \qquad \ \  + u(x)^{1-\frac{3}{2}\gamma(x)}\Delta u(x) \bigg),
\end{eqnarray*}
where
\begin{eqnarray*}
        a(x) & = & \displaystyle -\frac{3}{2}w(x)\ln(u(x))\Delta \gamma(x) + \frac{9}{4}w(x)(\ln(u(x)))^2|D\gamma(x)|^2\\
         &  & -\displaystyle  3u(x)^{1-\frac{3}{2}\gamma(x)} D\gamma(x)\cdot Du(x)\\
         &  & -\displaystyle 3\left(2-\frac{3}{2}\gamma(x)\right)\ln(u(x))u(x)^{1-\frac{3}{2}\gamma(x)} D\gamma(x) \cdot Du(x).
\end{eqnarray*}
Let us estimate $\Delta w$ from below. We use Lemma \ref{opt pt grad estimate} to obtain
\begin{align*}
	|a(x)| & \leq  2|D^2\gamma(x)|u(x)^{2-\frac{3}{2}\gamma(x)}|\ln(u(x))| + 3|D\gamma(x)|^2u(x)^{2-\frac{3}{2}\gamma(x)}|\ln(u(x))|^2\\
	&  \quad + C|D\gamma(x)|u(x)^{1-\gamma(x)} + C|D\gamma(x)|u(x)^{1-\gamma(x)}|\ln(u(x))|\\
	& \leq  C_1\left(|D\gamma(x)| + |D\gamma(x)|^2 +|D^2\gamma(x)|\right) \leq C_2
\end{align*}
for $C_2$ depending on $C_1$ and $\|D\gamma\|_\infty$ and $\|D^2 \gamma\|_\infty$, where we used that $u \leq 1$ and that the function $r^{\gamma'} \ln(r^{-1})$ is bounded for $\gamma' > 0$ and $r \leq 1$. To bound the second term in the expression for $\Delta w$, we split into two cases: if $1-\frac{3}{2}\gamma(x) \leq 0$, we use Lemma \ref{slightly dif pt grad est} to get a universal constant $C>0$ such that
\begin{eqnarray*}
    \displaystyle \left(\left(1-\frac{3}{2}\gamma(x)\right)u(x)^{-\frac{3}{2}\gamma(x)}|Du(x)|^2 + u(x)^{1-\frac{3}{2}\gamma(x)}\Delta u(x) \right) & = &   \\ 
     \displaystyle u(x)^{\frac{-\gamma(x)}{2}}\,\left(\left(1-\frac{3}{2}\gamma(x)\right)\frac{|Du(x)|^2}{u(x)^{\gamma(x)}} + \delta(x) \gamma(x)  \right) & \geq &\\
     \displaystyle u(x)^{\frac{-\gamma(x)}{2}}\,\left(\left(1-\frac{3}{2}\gamma(x)\right) \left(2\delta(x) + Cu(x)^{1-\gamma(x)}\right) + \delta(x) \gamma(x)  \right)& \geq & \\
     \displaystyle u(x)^{\frac{-\gamma(x)}{2}}\,\left(2\delta(x)(1-\gamma(x))  - 2Cu(x)^{1-\gamma(x)}  \right) & \eqqcolon & L(x). 
\end{eqnarray*}
We use Theorem \ref{thm-beicodebode} and that $\gamma^\ast(0,1) < 1$ to obtain
\begin{eqnarray*}
	L(x) & \geq & u(x)^{\frac{-\gamma(x)}{2}}\,\left(2\delta_0 \left(1-\gamma^\ast(0,1)\right) - 2C\tau^{\frac{2\left(1-\gamma^\ast(0,1)\right)}{2-\gamma(z)}} \right)\\
	  & \geq & \delta_0 \left(1-\gamma^\ast(0,1)\right) u(x)^{\frac{-\gamma(x)}{2}}, 
\end{eqnarray*}
provided $\tau$ is chosen small enough. In the case when $1-\frac{3}{2}\gamma(x) \geq 0$, we simply estimate
\begin{eqnarray*}
    \displaystyle \left(\left(1-\frac{3}{2}\gamma(x)\right)u(x)^{-\frac{3}{2}\gamma(x)}|Du(x)|^2 + u(x)^{1-\frac{3}{2}\gamma(x)}\Delta u(x) \right) & \geq &   \\ 
     \displaystyle \delta(x) \gamma(x) u(x)^{\frac{-\gamma(x)}{2}}& \geq &\\
     \displaystyle \delta_0 \gamma_\ast(0,1) u(x)^{\frac{-\gamma(x)}{2}}. & & 
\end{eqnarray*}
In any case, we obtain
$$
    \Delta w(x) \geq -C_2 + c u(x)^{-\frac{\gamma(x)}{2}},
$$
for a constant $c = c(\delta_0,\gamma^\ast(0,1),\gamma_\ast(0,1))$. Therefore, by \eqref{pen arg II}, we have
\begin{align*}
        A & \leq  \displaystyle C\tau^{n-1} +  \frac{1}{2}\int_{\{ \epsilon/4 < u \leq \epsilon\}\cap B_\tau(z_i)} \rho\left(\frac{2}{\epsilon} u\right) \left(C_2 - c u(x)^{-\frac{\gamma(x)}{2}}\right)\\
        &  \quad -\displaystyle  \int_{\{ \epsilon < u \}\cap B_\tau(z_i)} \left(-C_2 + cu(x)^{-\frac{\gamma(x)}{2}}\right)\\
        & \leq  \displaystyle C_3 \tau^{n-1} - c\int_{\{ \epsilon/4 < u \}\cap B_\tau(z_i)} u(x)^{-\frac{\gamma(x)}{2}},
\end{align*}
which gives
\begin{equation}\label{consequence eq above}
    A \leq C_3 \tau^{n-1} - c\int_{\{ \epsilon/4 < u \}\cap B_\tau(z_i)} u(x)^{-\frac{\gamma(x)}{2}}.
\end{equation}
Now, we estimate $A$ from below using \eqref{penalization argument}. By Lemma \ref{opt pt grad estimate} and since $\gamma^\ast(0,1) < 1$, we obtain
\begin{eqnarray*}
       Dw \cdot Du  & \geq & - 2u(x)^{2-\frac{3}{2}\gamma(x)} |\ln(u(x))| |D\gamma(x)|\,|Du(x)|\\
         & \geq & -Cu(x)^{2 - \gamma(x)} |\ln(u(x))| |D\gamma(x)|\\
         & \geq & -Cu(x)^{2 - \gamma^\ast(0,1)} |\ln(u(x))| |D\gamma(x)|\\
         & \geq & -C_4u(x)\, |D\gamma(x)|,
\end{eqnarray*}
for a larger constant $C_4$. Thus, from \eqref{penalization argument}, we have
$$
    -C_4 \frac{1}{\epsilon}\int_{\{0 \leq u < \epsilon \} \cap B_\tau(z)} u(x) |D\gamma(x)| \, \rho_\delta'(u)\,dx \leq A.
$$
Putting this together with \eqref{consequence eq above}, and taking into account that $\rho_\delta ' \leq 1$, we obtain
$$
   \int_{\{ \epsilon/4 < u \}\cap B_\tau(z)} u(x)^{-\frac{\gamma(x)}{2}} \leq C_5 \tau^{n-1}, 
$$
for a constant $C_5$ with the dependencies of the one in Lemma \ref{slightly dif pt grad est} and $\gamma^\ast(0,1)$. We get the result by passing to the limit as $\epsilon \rightarrow 0$.
\end{proof}

We are now ready to state and prove the main result of this section.

\begin{theorem}
Let $u$ be a local minimizer of the energy-functional \eqref{fifi} in $B_1$. Assume \eqref{bound delta}, \eqref{CMC}, \eqref{nunca derivei tanto}, and \eqref{gamma cima menor q 1} are in force. Then, there exists a universal constant $C>0$,
depending only on $n$, $\delta_0$, $\gamma_\ast (0,1)$, $\gamma^\ast (0,1)$, $\|D\delta\|_\infty$, $\|D^2 \delta\|_{\infty}$, $\|D\gamma\|_\infty$ and $\|D^2 \gamma\|_{\infty}$, such that
$$
\mathcal{H}^{n-1}(F(u) \cap B_{1/2}) < C.
$$
\end{theorem}

\begin{proof}
Assume $0 \in F(u)$. It is enough to prove that for small $r$,
$$
\mathcal{H}^{n-1}(F(u) \cap B_{r}) \leq Cr^{n-1}.
$$
Given a small parameter $\epsilon>0$, we cover $F(u) \cap B_r$ with finitely many balls $\{B_\epsilon(x_i) \}_{i \in F_\epsilon}$ with finite overlap, that is,
$$
    \sum_{i \in F_\epsilon} \chi_{B_\epsilon(x_i)} \leq c,
$$
for a constant $c>0$ that depends only on the dimension. It then follows that
\begin{equation}\nonumber
    \mathcal{H}^{n-1}(F(u)\cap B_r) \leq \overline{c} \liminf_{\epsilon \rightarrow 0} \epsilon^{n-1} \#(F_\epsilon).
\end{equation}
Since $x_i \in F(u)$, by Theorem \ref{thm-beicodebode}, we have
$$
    \Omega(u) \cap B_\epsilon(x_i) \subset \left\{0 < u \leq \overline{M}\epsilon^{\beta_i} \right\} \cap B_\epsilon(x_i),
$$
where $\beta_i = 2/(2-\gamma(x_i))$. By Assumption \eqref{CMC}, it follows that
$$
    \Omega(u) \cap B_\epsilon(x_i) \subset \left\{0 < u \leq \overline{M}_1\epsilon^{\beta^\ast(x_i,\epsilon)} \right\} \cap B_\epsilon(x_i),
$$
for a universal constant $\overline{M}_1 > \overline{M}$, with
$$
	\beta^\ast(x_i,\epsilon) \coloneqq \frac{2}{2-\gamma^\ast(x_i,\epsilon)}.
$$
Up to replacing $u$ by $u/\overline{M}_1$, we may assume $\overline{M}_1 = 1$. Now, observe that
$$
    \bigcup_{i 
    \in F_\epsilon} \left(B_\epsilon(x_i) \cap \left\{ 0 < u(x) \leq \epsilon^{\beta^\ast(x_i,\epsilon)} \right\} \right) \subseteq B_{2r} \cap \left\{0 < u(x)^{\frac{1}{\beta(x)}} < \epsilon \right\},
$$
with
$$
	\beta(x) \coloneqq \frac{2}{2-\gamma(x)}.
$$
Since the covering $\{B_\epsilon(x_i) \}_{i \in F_\epsilon}$ has finite overlap, it then follows that
$$
    \sum_{i \in F_\epsilon} \left| \Omega(u) \cap B_\epsilon(x_i) \right| \leq c \left| B_{2r} \cap \left\{0 < u(x)^{\frac{1}{\beta(x)}} < \epsilon \right\}\right|.
$$
From Theorem \ref{pos_density},
$$
    |\Omega(u) \cap B_\epsilon(x_i)| \geq \mu_0 \epsilon^n,
$$
and so
$$
     \epsilon^{n-1} \#(F_\epsilon) \leq \frac{c}{\mu_0} \frac{\left|B_{2r} \cap \left\{0 < u(x)^{\frac{1}{\beta(x)}} < \epsilon \right\}\right|}{\epsilon},
$$
which readily leads to
$$
    \mathcal{H}^{n-1}(F(u) \cap B_r) \leq \frac{\overline{c}\, c}{\mu_0} \liminf_{\epsilon \rightarrow 0} \frac{\left|B_{2r} \cap \left\{0 < u(x)^{\frac{1}{\beta(x)}} < \epsilon \right\}\right|}{\epsilon}.
$$
We will show below that the right-hand side of the inequality above can be bounded above uniformly in $\epsilon$. To do so, let 
$$
    v(x) \coloneqq u(x)^{\frac{1}{\beta(x)}}.
$$
Observe that
$$
    \int\limits_{B_{2r} \cap \{0 < v \leq \epsilon \}} |Dv|^2\,dx  = \int\limits_{B_{2r}} D(\min(v,\epsilon)) \cdot Dv\,dx \eqqcolon I.
$$
Integrating by parts, we get
$$
    I = - \int\limits_{B_{2r}} \min(v,\epsilon) \Delta v\,dx + \int\limits_{\partial B_{2r}} \min(v,\epsilon) \partial_\nu v\,d\mathcal{H}^{n-1} ,
$$
and so,
\begin{equation}\label{edu tomando uma domingo e nois derivando}
     \int\limits_{V_\epsilon} \left(|Dv|^2 + v\Delta v\right)\,dx = - \epsilon \int\limits_{B_{2r} \cap \{v > \epsilon \}}\Delta v\,dx + \int\limits_{\partial B_{2r}} \min(v,\epsilon) \partial_\nu v\,d\mathcal{H}^{n-1},
\end{equation}
where $V_\epsilon \coloneqq B_{2r} \cap \{v \leq \epsilon \}$. Let us bound the left-hand side of \eqref{edu tomando uma domingo e nois derivando} from below. By direct computations, we readily obtain
$$
    Dv(x)  =   g(x)D \left(\frac{1}{\beta(x)} \right) + \frac{1}{\beta(x)}u(x)^{\frac{1}{\beta(x)} - 1}Du(x)
$$
and
$$
    \Delta v(x)  =  A(x) + B(x) + \frac{\delta(x)\, \gamma(x)}{\beta(x)} u(x)^{-\frac{1}{\beta(x)}},
$$
where $g(x) = v(x) \ln(u(x))$, with 
\begin{eqnarray*}
A(x) \coloneqq g(x)\,\Delta\!\left(\tfrac{1}{\beta(x)}\right)
+ D\!\left(\tfrac{1}{\beta(x)}\right)\!\cdot\Big(Dg(x)+u(x)^{\frac{1}{\beta(x)}-1}Du(x)\Big).
\end{eqnarray*}
and
$$
    B(x) \coloneqq \frac{1}{\beta(x)} D\left(u^{\frac{1}{\beta(x)}-1}\right) \cdot Du(x).
$$
Now we estimate
\begin{eqnarray*}
|Dv|^2 + v \Delta v & = & \displaystyle \underbrace{g(x)^2\left|D\left(\frac{1}{\beta(x)}\right)\right|^2}_{\geq 0} + \frac{1}{\beta(x)^2}u(x)^{2\left(\frac{1}{\beta(x)}-1\right)}|Du|^2 \\
         & & +    2\frac{1}{\beta(x)}g(x)u(x)^{\frac{1}{\beta(x)} - 1}D\left(\frac{1}{\beta(x)}\right) \cdot Du(x)\\
         & & +  (A(x) + B(x))u(x)^{\frac{1}{\beta(x)}} + \frac{\delta(x)\, \gamma(x)}{\beta(x)}\\
         & \geq & \underbrace{\frac{1}{\beta(x)^2}u(x)^{2\left(\frac{1}{\beta(x)}-1\right)}|Du|^2 +  B(x) u(x)^{\frac{1}{\beta(x)}}}_{\mathcal{I}}\\
         & & +   2\frac{1}{\beta(x)}g(x)u(x)^{\frac{1}{\beta(x)} - 1}D\left(\frac{1}{\beta(x)}\right) \cdot Du(x)\\
         & & +   A(x)u(x)^{\frac{1}{\beta(x)}} + \frac{\delta_0 \, \gamma_\ast (0,1)}{2}.
\end{eqnarray*}
The worst term in the expression above is $B(x) u(x)^{\frac{1}{\beta(x)}}$, which is of order $u^{-1}$. To handle it, we make use of the following cancellation on the term $\mathcal{I}$:
\begin{eqnarray*}
	\mathcal{I} & = & \frac{1}{\beta(x)^2}u(x)^{2\left(\frac{1}{\beta(x)}-1\right)}|Du(x)|^2 + \frac{1}{\beta(x)}\left(\frac{1}{\beta(x)}-1 \right)u(x)^{\frac{2}{\beta(x)}-2}|Du(x)|^2\\
	& & - \frac{1}{2\beta(x)}u(x)^{\frac{2}{\beta(x)}-1}\ln(u) \left(D\gamma(x)\cdot Du(x)\right)\\
	& = & \frac{1}{\beta(x)}u(x)^{2\left(\frac{1}{\beta(x)}-1\right)}|Du(x)|^2\left(\frac{2}{\beta(x)} - 1 \right)\\
	& & - \frac{1}{2\beta(x)}u(x)^{\frac{2}{\beta(x)}-1}\ln(u) \left(D\gamma(x)\cdot Du(x)\right).
\end{eqnarray*}
Taking into account that 
$$
	\frac{2}{\beta(x)} - 1  = 1 - \gamma(x) \geq 0,
$$
we obtain
\begin{align*}
   \mathcal{I} & \geq -|D\gamma(x)|u(x)^{\frac{2}{\beta(x)}-1}|Du(x)||\ln(u(x))|\\
               & \geq - C |D\gamma(x)|u(x)^{\frac{2}{\beta(x)}-1},               
\end{align*}
where we used Lemma \ref{opt pt grad estimate} to bound the term $|Du(x)||\ln(u(x))|$. Therefore,
$$
	\mathcal{I} \geq -Cu(x)^{1-\gamma(x)}|D\gamma(x)|,
$$
where we have used Lemma \ref{opt pt grad estimate}. Putting everything together, we get
\begin{align*}
         |Dv|^2 + v \Delta v & \geq  \displaystyle  2\frac{1}{\beta(x)}g(x)u(x)^{\frac{1}{\beta(x)} - 1}D\left(\frac{1}{\beta(x)}\right) \cdot Du(x)\\
         & \quad - Cu^{1-\gamma(x)}|D\gamma(x)|  + A(x)u(x)^{\frac{1}{\beta(x)}} + \frac{\delta_0 \, \gamma_\ast (0,1)}{2},
\end{align*}
for some universal constant $C$. Using Lemma \ref{opt pt grad estimate} once more, we obtain
\begin{align*}
    \left|2 \frac{1}{\beta(x)} g(x)u(x)^{\frac{1}{\beta(x)} - 1} D\left(\frac{1}{\beta(x)} \right) \cdot Du(x)  \right|& \leq  C_1 u(x)^{1-\frac{3}{4}\gamma(x)} |D\gamma(x)|\\
    & \leq C_1 u(x)^{\frac{1}{2\beta(x)}}|D\gamma(x)|,
\end{align*}
and
\begin{align*}
	|A(x)| & \leq  |D^2\gamma(x)|u(x)^{\frac{1}{\beta(x)}}|\ln(u(x))| + |D\gamma(x)|^2u(x)^{\frac{1}{\beta(x)}}(\ln(u(x)))^2\\
	&  \quad + C_1|\ln(u(x))| + 2C_1 |D\gamma(x)|, 
\end{align*}
for some universal constant $C_1$, and so
\begin{eqnarray*}
         |Dv|^2 + v \Delta v & \geq & \displaystyle - C_2u^{\frac{1}{2\beta(x)}}\left(|D\gamma(x)| + |D\gamma(x)|^2 + |D^2\gamma(x)| + 1 \right)\\
         & &  + \frac{\delta_0 \, \gamma_\ast (0,1)}{2},
\end{eqnarray*}
for a larger constant $C_2$. Recalling that we are within the set $V_\epsilon$, we have $u(x)^{\frac{1}{2\beta(x)}} \leq \epsilon^{1/2}$, and so, for $\epsilon$ small enough depending further on $C_2$, $\|D\gamma\|_\infty$ and $\|D^2\gamma\|_\infty$, there holds
$$
	|Dv|^2 + v \Delta v \geq \frac{\delta_0 \, \gamma_\ast (0,1)}{4}.
$$
We can now estimate the left-hand side of \eqref{edu tomando uma domingo e nois derivando} as
\begin{eqnarray*}
         \int\limits_{B_{2r} \cap \{0 < v \leq \epsilon \} } \left(|Dv|^2 + v\Delta v\right)\,dx  & \geq &  \frac{\delta_0 \, \gamma_\ast (0,1)}{4}|B_{2r} \cap \{0 < v \leq \epsilon \}|.
\end{eqnarray*}
By Lemma \ref{opt pt grad estimate}, there exists a constant $C_3>0$ such that $|Dv| \leq C_3$, and so \eqref{edu tomando uma domingo e nois derivando} implies
$$
    \frac{\delta_0 \, \gamma_\ast (0,1)}{4}|B_{2r} \cap \{0 < v \leq \epsilon \}| \leq - \epsilon \int\limits_{B_{2r} \cap \{v > \epsilon \}}\Delta v\,dx + C_3 \epsilon r^{n-1},
$$
and so
$$
    \frac{\delta_0 \, \gamma_\ast (0,1)}{4}\frac{|B_{2r} \cap \{0 < v \leq \epsilon \}|}{\epsilon} \leq C_3r^{n-1} - \int\limits_{B_{2r} \cap \{v > \epsilon \}}\Delta v\,dx.
$$
The proof will then be complete provided this remaining integral is uniformly bounded in $\epsilon>0$. Recalling the estimate for $|A(x)|$, we have
$$
	|A(x)| \leq C_4|\ln(u(x))|,
$$
and 
$$
	-B(x) \leq C_4|\ln(u(x))| - \frac{1}{\beta(x)}\left(\frac{1}{\beta(x)} -1\right)u(x)^{\frac{1}{\beta(x)}-2}|Du(x)|^2,
$$
we have
\begin{align*}
        -\Delta v &  =   - A(x) - B(x) - \frac{\delta(x)\, \gamma(x)}{\beta(x)} u(x)^{-\frac{1}{\beta(x)}} \\
         & \leq   2C_4|\ln(u(x))| - \frac{1}{\beta(x)}\left(\frac{1}{\beta(x)} -1\right)u(x)^{\frac{1}{\beta(x)}-2}|Du(x)|^2 
         \\
         &  \quad - \frac{\delta(x)\, \gamma(x)}{\beta(x)} u(x)^{-\frac{1}{\beta(x)}}\\
         & \leq  2C_4|\ln(u(x))| - \frac{\delta(x)\, \gamma(x)}{\beta(x)} u(x)^{-\frac{1}{\beta(x)}}\\
         & \quad - \frac{1}{\beta(x)}\left(\frac{1}{\beta(x)} -1\right)u(x)^{\frac{1}{\beta(x)}-2}\left(2\delta(x)u(x)^{\gamma(x)} + c_1 u(x)\right)\\
         & =  2C_4|\ln(u(x))| + c_1u(x)^{\frac{1}{\beta(x)}-1} \leq C_5 u(x)^{-\frac{\gamma(x)}{2}},
\end{align*}
where we used Lemma \ref{slightly dif pt grad est} and the fact that $|\ln(u(x))|$ can be bounded above by $u(x)^{-\frac{\gamma(x)}{2}}$. This implies that
$$
  -\int\limits_{B_{2r} \cap \{v > \epsilon \}}\Delta v\,dx \leq C_5\int\limits_{B_{2r} \cap \{v > \epsilon \}}u(x)^{-\frac{\gamma(x)}{2}}\,dx. 
$$
Recalling the proof of Lemma \ref{integrability of negative power of u}, we have
$$
	\int\limits_{B_{2r} \cap \{v > \epsilon \}}u(x)^{-\frac{\gamma(x)}{2}}\,dx \lesssim r^{n-1},
$$
from which the conclusion of the theorem follows.
\end{proof}

\section{Monotonicity formula and classification of blow-ups}\label{sct monotonicity}

In this section, we derive a monotonicity formula valid for local minimizers of the energy-functional \eqref{fifi}, and we use it to classify blow-ups as homogeneous functions. We begin with the following definition.

\begin{definition}[Blow-up]\label{blowup definition}
Given a point $z_0 \in F(u)$, we say that $u_0$ is a blow-up of $u$ at $z_0$ if the family $\{u_r \}_{r>0}$, defined by 
$$
    u_r(x) \coloneqq \frac{u(z_0 + rx)}{r^{\beta(z_0)}}, \quad \text{with} \quad \beta(z_0) \coloneqq \frac{2}{2-\gamma(z_0)},
$$
converges, along a subsequence, to $u_0$, when $r\to 0$.

We say $u_0$ is $\beta(z_0)$-homogeneous if
$$
	u_0(\lambda x) = \lambda^{\beta(z_0)}u_0(x), \quad \forall \lambda > 0,  \ \forall x \in \mathbb{R}^n.
$$
\end{definition}

The construction of this new monotonicity formula is based on the behavior of the functional \eqref{fifi} under functions that are already homogeneous.

\begin{lemma}\label{bhv with homog functions}
Let $z_0 \in \mathbb{R}^n$ and $v \in C^{0,1}(B_1)$ be a $\beta(z_0)$-homogeneous function. Define
$$
	\beta_0 \coloneqq \beta(z_0) \quad \text{and} \quad \gamma_0 \coloneqq \gamma(z_0).
$$
For $r>0$ such that $B_r(z_0) \subset B_1$, define the quantity $\mathcal{H}_{v,z_0}(r)$ by

\begin{eqnarray}
\displaystyle  &   & \displaystyle r^{-(n + 2(\beta_0-1))}\mathcal{J}_\delta^\gamma(v,B_r(z_0)) - \frac{1}{2}\beta_0 r^{-((n-1) + 2\beta_0)}\int\limits_{\partial B_r(z_0)}v^2\, d\mathcal{H}^{n-1} \nonumber\\
&  & - \int_{0}^{r} \beta_0 t^{-(n+\beta_0 \gamma_0 +1)}\left(\int\limits_{B_t(z_0)}(\gamma(x)-\gamma_0)\delta(x)v^{\gamma(x)} \, dx\right)\, dt  \nonumber \\
&  & - \int_{0}^{r} t^{-(n+\beta_0 \gamma_0 + 1)}\left(\int\limits_{B_t(z_0)}\left(D\gamma(x)\cdot (x-z_0)\right) \delta(x) v^{\gamma(x)}  \ln(v) \, dx\right)\, dt \nonumber \\
&  & -  \int_{0}^r t^{-(n+\beta_0 \gamma_0 + 1)} \left(\int\limits_{B_t(z_0)} (D\delta(x) \cdot (x - z_0)) v^{\gamma(x)}  \, dx\right)\, dt.\label{mon formula}
\end{eqnarray}
Then, it follows that 
$$
	\frac{d}{dr} \mathcal{H}_{v,z_0}(r) = 0.
$$
\end{lemma} 

\begin{proof}
Without loss of generality, we may assume $z_0 = 0$. Define
$$
  \overline{\mathcal{H}}(r) \coloneqq   r^{-(n + 2(\beta_0-1))}\mathcal{J}_\delta^\gamma(v,B_r) - \frac{1}{2}\beta_0 r^{-((n-1) + 2\beta_0)}\int\limits_{\partial B_r}v^2\, d\mathcal{H}^{n-1}.
$$
Since $v$ is $\beta_0$-homogeneous, changing variables allows us to write
$$
	\overline{\mathcal{H}}(r) = \int_{B_1} \frac{1}{2}|Dv|^2 + \delta(rx)r^{\beta_0(\gamma(rx)-\gamma_0)}v(x)^{\gamma(rx)}\,dx -\frac{1}{2}\beta_0 \int_{\partial B_1}v^2\,d\mathcal{H}^{n-1},
$$
and so
$$
	\frac{d}{dr}\overline{\mathcal{H}}(r) = \int_{B_1} \frac{d}{dr}\left(\delta(rx)r^{\beta_0(\gamma(rx)-\gamma_0)}v(x)^{\gamma(rx)} \right)dx.
$$
Direct computations now give
\begin{eqnarray*}
	\frac{d}{dr}\overline{\mathcal{H}}(r) & = & \int_{B_1} \left(D\delta(rx)\cdot x \right)r^{\beta_0(\gamma(rx)-\gamma_0)}v(x)^{\gamma(rx)}\,dx\\
	& & + \int_{B_1}\delta(rx)\left(\beta_0(\gamma(rx)-\gamma_0)r^{\beta_0(\gamma(rx)-\gamma_0)-1}v^{\gamma(rx)} \right)\,dx\\
	& & + \underbrace{\int_{B_1}\delta(rx)\left(\beta_0(D\gamma(rx) \cdot x)r^{\beta_0(\gamma(rx)-\gamma_0)}\ln(r)v^{\gamma(rx)} \right)\,dx}_{I}\\
	& & + \underbrace{\int_{B_1}\delta(rx)\left((D\gamma(rx) \cdot x)r^{\beta_0(\gamma(rx)-\gamma_0)}\ln(v)v^{\gamma(rx)} \right)\,dx}_{II}.
\end{eqnarray*}
By the $\beta_0$-homogeneity of $v$, the last two terms can be summed, which gives
\begin{eqnarray*}
	I + II =  \int_{B_1}\delta(rx)\left((D\gamma(rx) \cdot x)r^{\beta_0(\gamma(rx)-\gamma_0)}\ln(v(rx))v^{\gamma(rx)} \right)\,dx.
\end{eqnarray*}
Changing variables back and using again the homogeneity of $v$, we obtain
\begin{eqnarray*}
	\frac{d}{dr}\overline{\mathcal{H}}(r) & = & r^{-n-\beta_0 \gamma_0-1}\int_{B_r} \left(D\delta(x)\cdot x \right)v(x)^{\gamma(x)}\,dx\\
	& & + \beta_0r^{-n-\beta_0\gamma_0-1}\int_{B_r}\delta(x)(\gamma(x)-\gamma_0)v^{\gamma(x)}\,dx\\
	& & + r^{-n-\beta_0\gamma_0-1}\int_{B_r}\delta(x)\left(D\gamma(x) \cdot x\right)\ln(v(x))v^{\gamma(x)}\,dx,
\end{eqnarray*}
from which follows that $\frac{d}{dr}\mathcal{H}_{v,0}$ is zero.
\end{proof}

We require further regularity assumptions on both $\delta$ and $\gamma$ to ensure the quantity $\mathcal{H}_{v,z_0}(r)$ is finite. To that end, we need, for some $r_0 \in (0,1/2)$, that
\begin{equation}\label{integra delta}
    t \mapsto t^{-n}\int_{B_t(z_0)}|D\delta(x)|\, dx \in L^1(0,r_0),
\end{equation}
and
\begin{equation}\label{integra gamma}
    t \mapsto t^{-n}|\ln t| \int_{B_t(z_0)}|D\gamma(x)|\, dx \in L^1(0,r_0).
\end{equation}
We remark that sufficient conditions for these to hold are $|D\delta| \in L^{q}(B_1)$ and $|D\gamma| \in L^{q}(B_1)$, for $q> n$. Indeed, we readily have
$$
    t^{-n}|\ln t|\int_{B_t(z_0)}|D\gamma(x)|\, dx \leq C(n,q) \|D\gamma\|_{L^q \left( B_{r_0}(z_0) \right)} \, t^{-\frac{n}{q}}|\ln t|,
$$
and
$$
    \int_0^{r_0} t^{-\frac{n}{q}}|\ln t|\, dt < \infty \quad \Longleftarrow \quad q>n.
$$

\begin{remark}
If we assume $\gamma \in W^{1,q}$, for $q>n$, then $\gamma$ is H\"older continuous and therefore condition \eqref{CMC} is automatically satisfied. We also point out that these integrability conditions are important to ensure that $\mathcal{H}_{u,z_0}(r) < \infty$, for every $r>0$, and $z_0 \in F(u)$ such that $B_r(z_0) \Subset B_1$, for $u$ a local minimizer of \eqref{fifi}.
\end{remark}

Following the strategy of the proof of the monotonicity formula in \cite{DSS2}, we apply Lemma \ref{bhv with homog functions} for a specific $\beta_0$-homogeneous function (namely the $\beta_0$-homogeneous extension of $u$) and use it as a competitor to \eqref{fifi} to prove that a version of the formula \eqref{mon formula} is monotone in $r$.

\begin{theorem}\label{monotonicity formula}
Let $u$ be a local minimizer of \eqref{fifi} and assume \eqref{integra delta} and \eqref{integra gamma} are in force. For $z_0 \in F(u)\cap B_{1/2}$, define
$$
	\beta_0 \coloneqq \beta(z_0) \quad \text{and} \quad \gamma_0 \coloneqq \gamma(z_0),
$$
and for $r \in (0,r_0)$, consider the function $\mathcal{W}_{u,z_0}(r)$ defined by
\begin{eqnarray*}
  &   &  r^{-(n + 2(\beta_0-1))}\mathcal{J}_\delta^\gamma(u,B_r(z_0)) - \frac{1}{2}\beta_0 r^{-((n-1) + 2\beta_0)}\int\limits_{\partial B_r(z_0)}u^2\, d\mathcal{H}^{n-1}\\
&  & - \int_{0}^{r} \beta_0 t^{-(n+\beta_0 \gamma_0 +1)}\left(\int\limits_{B_t(z_0)}(\gamma(x)-\gamma_0)\delta(x)w^{\gamma(x)} \, dx\right)\, dt \\
&  & - \int_{0}^{r} t^{-(n+\beta_0 \gamma_0 + 1)}\left(\int\limits_{B_t(z_0)}\left(D\gamma(x)\cdot (x-z_0)\right) \delta(x) w^{\gamma(x)}  \ln(w) \, dx\right)\, dt\\
&  & -  \int_{0}^r t^{-(n+\beta_0 \gamma_0 + 1)} \left(\int\limits_{B_t(z_0)} (D\delta(x) \cdot (x - z_0)) w^{\gamma(x)}  \, dx\right)\,dt,
\end{eqnarray*}
where
$$
    w(x,t) \coloneqq \left(\frac{|x-z_0|}{t}\right)^{\beta_0} u\!\left(z_0 + t\frac{x-z_0}{|x-z_0|}\right),\quad x\ne z_0,\qquad w(z_0,t)=0.
$$
satisfies
$$
    \frac{d}{dr}\mathcal{W}_{u,z_0}(r) \geq 0.
$$
\end{theorem}

\begin{proof}
Without loss of generality, we may assume $z_0 = 0$. Define
$$
  \overline{\mathcal{H}}(r) \coloneqq   r^{-(n + 2(\beta_0-1))}\mathcal{J}_\delta^\gamma(u,B_r) - \frac{1}{2}\beta_0 r^{-((n-1) + 2\beta_0)}\int\limits_{\partial B_r}u^2\, d\mathcal{H}^{n-1}.
$$
By direct computations, we can write $\frac{d}{dr}\overline{\mathcal{H}}(r)$ as
\begin{eqnarray*}
	  &   -(n + 2(\beta_0-1))r^{-(n + 2\beta_0-1)}\mathcal{J}_\delta^\gamma(u,B_r) + r^{-(n + 2(\beta_0-1))}\mathcal{J}_\delta^\gamma(u,\partial B_r)\\
	 & \displaystyle - \beta_0r^{-(n+2\beta_0)}\int_{\partial B_r}(ru\partial_\nu u - \beta_0 u^2)d\mathcal{H}^{n-1}.
\end{eqnarray*}
Here, we are abusing notation and writing
$$
  \mathcal{J}_\delta^\gamma(u,\partial B_r) = \int_{\partial B_r}\frac{1}{2}|Du|^2 + \delta(x)u^{\gamma(x)}\,d\mathcal{H}^{n-1}.  
$$
For points in $\partial B_r$, we decompose $|Du(x)|^2$ into the tangential and normal components, 
$$
	Du(x) = \underbrace{(Du(x) \cdot \nu)}_{\partial_\nu u(x)} \nu + \underbrace{(Du(x) - (Du(x) \cdot \nu) \nu)}_{D_\tau u(x)},
$$
which gives the following expression for $\frac{d}{dr}\overline{\mathcal{H}}(r)$,
\begin{align*}
	  &  r^{-(n + 2(\beta_0-1))}\int_{\partial B_r}\left(\frac{1}{2}|D_\tau u|^2 + \delta(x)u^{\gamma(x)}\right)\\
	  & -(n + 2(\beta_0-1))r^{-(n + 2\beta_0-1)}\mathcal{J}_\delta^\gamma(u,B_r) + r^{-(n + 2(\beta_0-1))}\int_{\partial B_r}\frac{1}{2}\left(\partial_\nu u\right)^2\\
	  & - \beta_0r^{-(n+2\beta_0)}\int_{\partial B_r}(ru\partial_\nu u - \beta_0 u^2),
\end{align*}
where we intentionally omitted the $d\mathcal{H}^{n-1}$ from the surface integrals to ease notation. We can put together the last two integrals in the sphere to obtain the following expression for $\frac{d}{dr}\overline{\mathcal{H}}(r)$:
\begin{eqnarray*}
	  \displaystyle  \overbrace{r^{-(n + 2(\beta_0-1))}\int_{\partial B_r}\left(\frac{1}{2}|D_\tau u|^2 + \delta(x)u^{\gamma(x)}\right) + \frac{1}{2}\beta_0^2r^{-(n+2\beta_0)}\int_{\partial B_r}u^2}^{\mathcal{B}(u)}\\
	  \displaystyle    -(n + 2(\beta_0-1))r^{-(n + 2\beta_0-1)}\mathcal{J}_\delta^\gamma(u,B_r)\\
	  \displaystyle + \frac{1}{2} r^{-(n+2(\beta_0-1))}\int_{\partial B_r}(\partial_\nu u - \beta_0r^{-1} u)^2.
\end{eqnarray*}
Recalling the definition of $\mathcal{W}_{u,0}(r)$, we have
\begin{eqnarray*}
  \frac{d}{dr}\mathcal{W}_{u,0}(r)& =  & \mathcal{B}(u) -(n + 2(\beta_0-1))r^{-(n + 2\beta_0-1)}\mathcal{J}_\delta^\gamma(u,B_r)\\
  & &  - \beta_0 r^{-(n+\beta_0 \gamma_0 +1)}\int\limits_{B_r}(\gamma(x)-\gamma_0)\delta(x)w(x,r)^{\gamma(x)} \,dx\\
&  & - r^{-(n+\beta_0 \gamma_0 + 1)}\int\limits_{B_r}\left(D\gamma(x)\cdot x\right) \delta(x) w(x,r)^{\gamma(x)}  \ln(w(x,r)) \, dx\\
&  & -  r^{-(n+\beta_0 \gamma_0 + 1)}\int\limits_{B_r} (D\delta(x) \cdot x) w(x,r)^{\gamma(x)}  \, dx\\
& & + \frac{1}{2} r^{-(n+2(\beta_0-1))}\int_{\partial B_r}(\partial_\nu u - \beta_0r^{-1} u)^2.
\end{eqnarray*}
We now show that 
\begin{eqnarray}
  \mathcal{B}(u) & \geq &  (n + 2(\beta_0-1))r^{-(n + 2\beta_0-1)}\mathcal{J}_\delta^\gamma(u,B_r) \nonumber\\ 
  & &  + \beta_0 r^{-(n+\beta_0 \gamma_0 +1)}\int\limits_{B_r}(\gamma(x)-\gamma_0)\delta(x)w(x,r)^{\gamma(x)} \,dx \nonumber\\ 
&  & + r^{-(n+\beta_0 \gamma_0 + 1)}\int\limits_{B_r}\left(D\gamma(x)\cdot x\right) \delta(x) w(x,r)^{\gamma(x)}  \ln(w(x,r)) \, dx \nonumber\\ 
&  & +  r^{-(n+\beta_0 \gamma_0 + 1)}\int\limits_{B_r} (D\delta(x) \cdot x) w(x,r)^{\gamma(x)}  \, dx. \label{main ineq for monotonicity}
\end{eqnarray}
To do so, notice that
$$
	v(x) \coloneqq w(x,r) = \left(\frac{|x|}{r}\right)^{\beta_0}u\left(r\frac{x}{|x|}\right),
$$
is the $\beta_0$-homogeneous extension of $u$ to $\partial B_r$. Since $u = v$ on $\partial B_r$, we have $\mathcal{B}(u) = \mathcal{B}(v)$. Moreover, since $v$ is $\beta_0$-homogeneous, it follows by Lemma \ref{bhv with homog functions} that $\frac{d}{dr}\mathcal{H}_{v,0}(r) = 0$, where $\mathcal{H}_{v,0}(r)$ is the formula given by \eqref{mon formula}.  It also follows, by homogeneity of $v$, that $\mathcal{H}_{v,0}(r) = \mathcal{W}_{v,0}(r)$, and so it follows that inequality \eqref{main ineq for monotonicity} is an equality when exchanging $u$ by $v$. Take also into account that the integral
$$
	\int_{\partial B_r}(\partial_\nu v - \beta_0 r^{-1}v)^2\,d\mathcal{H}^{n-1} = 0,
$$
again by the homogeneity of $v$. Using $v$ as a competitor to the functional that $u$ minimizes, gives
$$
	\mathcal{J}_\delta^\gamma(u,B_r) \leq \mathcal{J}_\delta^\gamma(v,B_r).
$$
Putting everything together, we obtain \eqref{main ineq for monotonicity}.
\end{proof}

As a consequence of the monotonicity formula, we obtain the homogeneity of blow-ups. Unlike in the constant case $\gamma(x) \equiv \gamma_0$, the homogeneity property of blow-ups will vary depending on the free boundary point we are considering. This is the object of the following result.

\begin{corollary}\label{classification of blow ups}
Let $u$ be a local minimizer of \eqref{fifi} and assume \eqref{integra delta} and \eqref{integra gamma} are in force. If $u_0$ is a blow-up of $u$ at a point $z_0 \in F(u) \cap B_{1/2}$, then $u_0$ is $\beta(z_0)$-homogeneous.
\end{corollary}
\begin{proof}
Without loss of generality, we assume $z_0 = 0$. Recall
$$
    \beta_0 \coloneqq  \frac{2}{2-\gamma_0}, \quad \text{where} \quad \gamma_0 \coloneqq \gamma(0).
$$
By Definition \ref{blowup definition}, there is a sequence $\lambda_j \to 0$ such that
$$
	u_j(x) \coloneqq \frac{u(\lambda_j x)}{\lambda_j^{\beta_0}} \to u_0,\, \text{locally uniform in }\, \mathbb{R}^n.
$$
On one hand, Theorem \ref{monotonicity formula} ensures that
\begin{equation}\label{formula is monotone}
	\lim_{j \to \infty} \mathcal{W}_{u,0}(\lambda_j r) = \mathcal{W}_{u,0}(0^+),
\end{equation}
which follows from monotonicity. On the other hand, one can scale the formula in the parameter $\lambda_j$ to obtain the following expression for $\mathcal{W}_{u,0}(\lambda_j r)$:
\begin{eqnarray*}
	 &  & r^{-(n+2(\beta_0-1))}\mathcal{J}_{\delta_j}^{\gamma_j}(u_j,B_r)-\beta_0 r^{-((n-1) + 2\beta_0)}\int\limits_{\partial B_r}u_j^2\, d\mathcal{H}^{n-1}\\
	&  & - \underbrace{\int_{0}^{\lambda_j r} \beta_0 t^{-(n+\beta_0 \gamma_0 +1)}\left(\int\limits_{B_t}(\gamma(x)-\gamma_0)\delta(x)w(x,t)^{\gamma(x)} \, dx\right)\, dt}_{\mathcal{I}^j_1} \\
    &  & - \underbrace{\int_{0}^{\lambda_j r} t^{-(n+\beta_0 \gamma_0 + 1)}\left(\int\limits_{B_t}\left(D\gamma(x)\cdot x\right) \delta(x) w(x,t)^{\gamma(x)}  \ln(w(x,t)) \, dx\right)\, dt}_{\mathcal{I}^j_2}\\
    &  & -  \underbrace{\int_{0}^{\lambda_j r} t^{-(n+\beta_0 \gamma_0 + 1)} \left(\int\limits_{B_t} (D\delta(x) \cdot x)w(x,t)^{\gamma(x)} \, dx\right)\,dt}_{\mathcal{I}^j_3},
\end{eqnarray*}
where
$$
	\gamma_j(x) \coloneqq \gamma(\lambda_j x), \text{ and }\, \delta_j(x) \coloneqq \lambda_j^{\beta_0(\gamma_j(x)-\gamma_0)}\delta(\lambda_j x), 
$$
and $w$ is as defined in the statement of Theorem \ref{monotonicity formula}. Let us now show that the error terms (integrals $\mathcal{I}^j_1$, $\mathcal{I}^j_2$ and $\mathcal{I}^j_3$) tend to zero as $\lambda_j \to 0$. This can be justified by combining the dominated convergence theorem with growth estimates for $u$ and assumptions \eqref{integra delta}, \eqref{integra gamma}. Indeed, observe that by Theorem \ref{thm-beicodebode}, it follows that
\begin{eqnarray*}
    w(x,t)  & = & \left(\frac{|x|}{t}\right)^{\beta_0}u\left(t\frac{x}{|x|}\right)\\
            & \leq & Ct^{\beta_0},
\end{eqnarray*}
for $x \in B_t$. Using the $\mu-$H\"older continuity of $\gamma$ (which follows from the assumption that $\gamma \in W^{1,q}$ for $q>n$), we have
\begin{eqnarray*}
    \int\limits_{B_t}|\gamma(x)-\gamma_0|\delta(x)w(x,t)^{\gamma(x)} \, dx & \leq & Ct^{\mu} \int_{B_t} t^{\beta_0 \gamma(x)}\, dx\\
    & \leq & C t^{n + \mu + \beta_0 \gamma_\ast(0,t)}\\
    & \leq & C_1 t^{n + \mu + \beta_0 \gamma_0},
\end{eqnarray*}
where we used that $t^{\beta_0(\gamma_\ast(0,t)-\gamma_0)}$ is uniformly bounded as $t \to 0$. Therefore, it follows that the function
$$
    \beta_0 t^{-(n+\beta_0 \gamma_0 +1)}\left(\int\limits_{B_t}(\gamma(x)-\gamma_0)\delta(x)w(x,t)^{\gamma(x)} \, dx\right) \lesssim t^{\mu -1} \in L^1(0,1).
$$
By the dominated convergence theorem, $\mathcal{I}_1^j \to 0$ as $j \to \infty$. For the term $\mathcal{I}^j_2$, we have
\begin{eqnarray*}
	\left|\mathcal{I}^j_2 \right| & \leq & C_2 \int_{0}^{\lambda_j r}t^{-(n+\beta_0 \gamma_0)}\left(\int\limits_{B_t}|D\gamma(x)| t^{\beta_0 \gamma(x)}  |\ln(t^{\beta_0})| \, dx\right)\, dt\\
	& \leq & C_3 \int_{0}^{\lambda_j r}t^{-n}|\ln(t)|\left(\int\limits_{B_t}|D\gamma(x)| \, dx\right)\, dt.
\end{eqnarray*}
By assumption \eqref{integra gamma}, we can apply dominated convergence once again and obtain $\mathcal{I}_2^j \to 0$ as $j \to \infty$. The estimate for the term $\mathcal{I}^j_3$ follows the same lines of reasoning by using \eqref{integra delta} instead. Now, taking into account the following set of convergences
$$
	u_j \to u_0, \quad \delta_j \to \delta_0, \quad \gamma_j \to \gamma_0, \quad \text{and }\, \lambda_j^{\beta_0(\gamma_j - \gamma_0)} \to 1,
$$
locally uniform as $j \to \infty$, we obtain
$$
	\mathcal{W}^\infty_{u_0,0}(r) = \lim_{j \to \infty} \mathcal{W}_{u,0}(\lambda_j r),
$$
where
$$
	\mathcal{W}^\infty_{u_0,0}(r) \coloneqq r^{-(n+2(\beta_0-1))}\mathcal{J}_{\delta_0}^{\gamma_0}(u_0,B_r)-\frac{1}{2}\beta_0 r^{-((n-1) + 2\beta_0)}\int\limits_{\partial B_r}u_0^2\, d\mathcal{H}^{n-1}.
$$
Putting this together with \eqref{formula is monotone}, we have
$$
	\mathcal{W}^\infty_{u_0}(r) = \mathcal{W}_{u,0}(0^+),	
$$
for any $r>0$.

We conclude that $\mathcal{W}^\infty_{u_0,0}$ is constant, and since $u_0$ is a minimizer of the functional 
\begin{equation}\label{AP constant gamma}
    \mathcal{J}_{\delta_0}^{\gamma_0}(v,B_R) = \int_{B_R} \frac{1}{2}|Dv|^2 + \delta(0)v^{\gamma(0)} \, dx,
\end{equation}
for every $R>0$ (the proof of which follows the reasoning in \cite[Lemma 2.3 and Remark 1]{PT24}), it follows that it is $\beta_0$-homogeneous. This can be seen from the classical proof \cite[Lemma 7.1]{AP}, but also from the proof of Theorem \ref{monotonicity formula}, where we obtain that the quantity
$$
	\int_{\partial B_r}(\partial_\nu u_0 - \beta_0 r^{-1}u_0)^2\,d\mathcal{H}^{n-1}
$$
must be equal to zero for any $r>0$. This is equivalent to saying that $u_0$ is $\beta_0$-homogeneous.
\end{proof}

\begin{remark}\label{remark de convergencia}
To ensure the existence of blow-ups, one needs to guarantee that the family $(u_r)_{r>0}$, defined as
$$
    u_r(x) = \frac{u(z_0 + rx)}{r^{\beta(z_0)}} \quad \text{for} \quad \beta(z_0) = \frac{2}{2-\gamma(z_0)},
$$
is locally bounded in $C^{1,\beta(z_0)-1}$. Indeed, by Theorem \ref{thm-beicodebode}, there exists a constant $C'>1$ such that
$$
    \|u_r\|_{L^\infty(B_1)} \leq C'.
$$
Moreover, by applying Theorem \ref{localregthm} to $u$ over $B_{r}(z_0)$, we obtain
$${\mathrm{osc}}_{B_r(z_0)} |Du| \coloneqq \left(\sup_{B_r(z_0)} |Du|\right) - \left(\inf_{B_r(z_0)} |Du|\right) \leq C r^{\frac{\gamma_\ast(z_0,2r)}{2-\gamma_\ast(z_0,2r)}}.
$$
Proceeding as at the end of the proof of Theorem \ref{thm-beicodebode}, we use condition \eqref{CMC} to obtain
$$
    Cr^{\frac{\gamma_\ast(z_0,2r)}{2-\gamma_\ast(z_0,2r)}} \leq \overline{C} r^{\frac{\gamma(z_0)}{2-\gamma(z_0)}},
$$
which implies
$$
    {\mathrm{osc}}_{B_r(z_0)}|Du| \leq \overline{C} r^{\frac{\gamma(z_0)}{2-\gamma(z_0)}}.
$$
As a consequence, the family $\{u_r\}_{r>0}$ is locally bounded in $C^{1,\beta(z_0)-1}$.
\end{remark}

Given the above, blow-up limits of minimizers of the variable singularity functional \eqref{fifi} are global minimizers of an energy-functional with constant singularity, namely $\gamma(z_0)$. Corollary \ref{classification of blow ups} further yields that blow-ups are $\beta(z_0)$-homogeneous. 

The pivotal insight here is that the blow-up limits of minimizers of the variable singularity functional are entitled to the same theoretical framework applicable to the constant coefficient case. In particular, in dimension $n=2$, blow-up profiles are thoroughly classified due to \cite[Theorem 8.2]{AP}. More precisely, if $u_0$ is the blow-up of $u$ at $z_0 \in F(u)$, for $u$ a local minimizer of \eqref{fifi} and $0 < \gamma(z_0) < 1$, then $u_0$ verifies
$$
    \varrho_0(z_0)^{-\frac{1}{\beta(z_0)}}u_0(x)^{\frac{1}{\beta(z_0)}} =  (x \cdot \nu)_+ \quad \text{for} \quad x \in \mathbb{R}^n, 
$$
for some $\nu \in \partial B_1$, for a constant $\varrho_0(z_0)$, depending on $z_0$, precisely defined in the upcoming section.

\begin{definition}\label{def: min cone}
A minimizer $u$ of the energy-functional \eqref{fifi} with $\delta \equiv \delta(z_0)$ and $\gamma \equiv \gamma(z_0)$ for some $z_0 \in \mathbb{R}^n$ which is $\beta(z_0)$-homogeneous is called a $\beta(z_0)$-minimal cone. 
\end{definition}

Classifying minimal cones in lower dimensions is crucial, chiefly because of Federer's dimension reduction argument, which we will utilize in our upcoming section.

\section{Free boundary regularity}\label{sct FB reg}

In this final section, we investigate the regularity of the free boundary. For models with constant exponent $\gamma$, differentiability of the free boundary was obtained in \cite{AP}, following the developments of \cite{AC}. Although it may seem plausible, the task of amending the arguments from \cite{AC, AP} to the case of varying exponents -- the object of study of this paper -- proved quite intricate. More recently, similar free boundary regularity estimates have been obtained via a linearization argument in \cite{DSS} (see also \cite{DS}). Here, we will adopt the latter strategy, \textit{i.e.}, and proceed through an approximation technique, where the tangent models are the ones with constant $\gamma$.

More precisely, given a point $z_0 \in F(u)$, let us define
\[
\beta(z_0) \coloneqq \frac{2}{2 - \gamma(z_0)}, \qquad
\varrho(z_0) \coloneqq \left( \frac{(\beta(z_0)-1)\,\beta(z_0)}{\gamma(z_0)\,\delta(z_0)} \right)^{\tfrac{1}{\gamma(z_0)-2}}
\]
and
$$
    w = \varrho(z_0)^{-\frac{1}{\beta(z_0)}} u^{\frac{1}{\beta(z_0)}}.
$$
We note that since the equation holds within the set where $u$ is positive, we have
$$
    \delta(x)\, \gamma(x) u^{\gamma(x)-1} = \varrho(z_0) \beta(z_0) w^{\beta(z_0)-2} \left[w \Delta w + (\beta(z_0)-1)|Dw|^2 \right],
$$
and so
$$
    w\Delta w = \delta(x)\frac{\gamma(x)}{\beta(z_0)}\varrho(z_0)^{\gamma(x)-2}w^{\beta(z_0)(\gamma(x)-1)+2-\beta(z_0)} - (\beta(z_0)-1)|Dw|^2.
$$
Since
$$
    \beta(z_0)(\gamma(x)-1)+2-\beta(z_0) = \beta(z_0)(\gamma(x) - \gamma(z_0)),
$$
we can rewrite the equation as
\begin{equation}\label{eq almost linearized}
    \Delta w = \frac{h(x,w,Dw)}{w},
\end{equation}
where $h\colon B_1 \times \mathbb{R}^+ \times \mathbb{R}^n \mapsto \mathbb{R}$ is defined as
\begin{equation}\label{RHS of almost linear eq}
    h(x,s,\xi) \coloneqq \delta(x)\frac{\gamma(x)}{\beta(z_0)}
  \,\varrho(z_0)^{\gamma(x)-2}\,s^{\beta(z_0)\,(\gamma(x)-\gamma(z_0))}
  - \bigl(\beta(z_0)-1\bigr)\,|\xi|^2.
\end{equation}
It will be useful to the upcoming analysis to introduce the two related distorted functions
\begin{equation}\label{extremal functions}
    w^+ = \varrho(z_0)^{-\frac{1}{\beta(z_0)}} u^{\frac{1}{\beta^\ast}} \quad \text{and} \quad w^- = \varrho(z_0)^{-\frac{1}{\beta(z_0)}} u^{\frac{1}{\beta
    _\ast}},
\end{equation}
where
$$
   \beta^\ast \coloneqq \beta^\ast(z_0,1) \quad \text{and} \quad  \beta_\ast \coloneqq \beta_\ast(z_0,1).
$$
It is not hard to see, by the very same computations, that $w^+$ solves
$$
    w^+\Delta w^+ = \Lambda^+_{z_0}(x)(w^+)^{\beta^\ast(\gamma(x)-\gamma^\ast)} - (\beta^\ast - 1)|Dw^+|^2,
$$
in $\{w^+ >0\} = \{w>0\}$, where
$$
    \Lambda^+_{z_0}(x) \coloneqq \frac{\varrho(z_0)^{\frac{\beta^\ast}{\beta(z_0)}(\gamma(x)-2)}}{\beta^\ast}\delta(x)\gamma(x) \quad \text{and} \quad \gamma^\ast \coloneqq \sup_{B_1(z_0)} \gamma(x).
$$
Likewise
$$
    w^-\Delta w^- = \Lambda^-_{z_0}(x)(w^-)^{\beta_\ast(\gamma(x)-\gamma_\ast)} - (\beta_\ast - 1)|Dw^-|^2,
$$
in $\{w^- >0\} = \{w>0\}$, with
$$
    \Lambda^-_{z_0}(x) \coloneqq \frac{\varrho(z_0)^{\frac{\beta_\ast}{\beta(z_0)}(\gamma(x)-2)}}{\beta_\ast}\delta(x)\gamma(x) \quad \text{and} \quad \gamma_\ast \coloneqq \inf_{B_1(z_0)} \gamma(x).
$$
Since $\beta_\ast \leq \beta(z_0) \leq \beta^\ast$, we have
$$
    w^- \leq w \leq w^+,
$$
near the free boundary point $z_0$.

In \cite{DSS}, when $\gamma(\cdot)$ is constant, the function appearing on the right-hand side of \eqref{eq almost linearized} reduces essentially to $(1-|\xi|^2)$, which is nonnegative for $\xi \in \overline{B}_1$ and negative outside this region. This sign structure plays a crucial role in the construction of barriers. In our setting, however, the situation is fundamentally different: the region where the function $h(x,s,\xi)$ changes sign depends on the solution itself, which, \textit{a priori}, may render the construction of suitable barriers considerably more delicate. To overcome this issue, we exploit the fact that $w$ is trapped between $w^-$ and $w^+$, and the right-hand sides of the equations they satisfy have a definite sign (from below and above, respectively), which is crucial for the construction of strict sub- and supersolutions. Moreover, as the scale increases, these inequalities become increasingly sharp, allowing for the successful implementation of the recursive improvement of flatness strategy.

We first remark that defining $w_r$ as
\begin{equation}\label{lip rescaling of linearized eq}
    w_r(x) = \frac{w(z_0 + rx)}{r},
\end{equation}
direct calculations yield 
\begin{equation}\label{scaled equation}
    \Delta w_r = \frac{h_r(x,w_r,Dw_r)}{w_r},
\end{equation}
where
\begin{align*}
    h_r(x,s,\xi) & \coloneqq \delta(z_0 + rx)\frac{\gamma(z_0 + rx)}{\beta(z_0)}\varrho(z_0)^{\gamma(z_0 + rx) - 2} (rs)^{\beta(z_0)(\gamma(z_0 + rx) - \gamma(z_0))} \\
    & \quad - (\beta(z_0)-1)|\xi|^2. 
\end{align*}
We can now pass to the limit as $r \rightarrow 0$, and in view of the choice of $\varrho(z_0)$, we reach
$$
    h_r(x,s,\xi) \rightarrow  \overline{h}(z_0,\xi),
$$
where $\overline{h}(z_0,\xi)$ is given by
$$
    \overline{h}(z_0,\xi) \coloneqq (\beta(z_0)-1)(1-|\xi|^2).
$$

The second key remark is that if the exponent function $\gamma(x)$ is assumed to be H\"older continuous, say,  of order $\mu \in (0,1)$, then for a fixed $s>0$, the above convergence does not depend on the free boundary point, $z_0 \in F(u)$. Indeed, we can estimate
\begin{eqnarray*} 
|\beta(z_0)(\gamma(z_0 + rx) - \gamma(z_0))\ln(rs)|  & \leq  &\displaystyle  Cr^\mu|\ln(r) + \ln(s)|\\
     & \leq &  C([\gamma]_{C^{0,\mu}},|\ln(s)|) r^{\frac{\mu}{2}},
\end{eqnarray*}
which implies that
$$
    \lim_{r \rightarrow 0} (rs)^{\beta(z_0)(\gamma(z_0 + rx) - \gamma(z_0))} = 1.
$$
Arguing similarly, one also obtains that
$$
    \lim_{r \rightarrow 0} \delta(z_0 + rx)\frac{\gamma(z_0 + rx)}{\beta(z_0)}\varrho(z_0)^{\gamma(z_0 + rx) - 2} = \beta(z_0) - 1,
$$
uniformly in $z_0 \in F(u)$. Here, we only need the uniform continuity of the ingredients involved. These insights are critical to ensure the linearized problem is uniformly close to the one with constant exponent as treated in \cite{DSS} , and they will be made precise later.

A final remark concerns the notion of viscosity solution used to interpret the preceding equations and the free boundary condition. Assuming that $\gamma$ and $\delta$ are H\"older continuous, the extremal functions $w^-, w^+$ and $w$ are classical solutions on their positivity sets; hence they solve the PDE there in the usual viscosity sense. The delicate point is verifying the free boundary condition. For this, we introduce the following notation: given $x,y\in B_2$, define the quotient
\[
    \varphi(x,y) \coloneqq 
        \frac{\varrho(y)^{-1/\beta(y)}}{\varrho(x)^{-1/\beta(x)}}.
\]
Now consider $\mathcal{D}^\ast \colon B_1 \to \mathbb{R}^+$ and $\mathcal{D}_\ast \colon B_1 \to \mathbb{R}^+$ defined by
\[
    \mathcal{D}^\ast(x) \coloneqq 
        \sup_{y \in B_1(x)} \varphi(x,y),
    \qquad 
    \mathcal{D}_\ast(x) \coloneqq 
        \inf_{y \in B_1(x)} \varphi(x,y),
\]
which always satisfies $\mathcal{D}^\ast \geq 1 \geq \mathcal{D}_\ast$. Then, for each $x \in B_1$, we define the (possibly overlapping) regions
$$
    \Gamma_\ast(x) \coloneqq \overline{B}_{\mathcal{D}^\ast(x)} \cap \mathbb{R}^n \setminus B_{\mathcal{D}_\ast(x)}.
$$
As a check, when $\gamma$ and $\delta$ are constant we recover the classical situation: $\Gamma_\ast(x) = \partial B_1$ for every $x \in B_1$.

\begin{definition}\label{sub-super solution property}
Let $\overline{w}$ be nonnegative. We say that 
$$
D\overline{w}\in \Gamma_\ast \quad \text{on} \quad F(\overline{w})\cap B_1
$$ 
\emph{in the viscosity sense} if, 
for every $z_0\in F(\overline{w})\cap B_1$ and every $\psi\in C^2$ such that $\psi^+$ touches $\overline{w}$ 
from below (resp.\ from above) at $z_0$ with $|D\psi(z_0)|\neq 0$, one has
\[
\begin{aligned}
|D\psi(z_0)| &\le \mathcal{D}^\ast(z_0) \quad (\text{resp.\ } |D\psi(z_0)| \ge \mathcal{D}_\ast(z_0)),\\
\text{that is}\quad 
D\psi(z_0) &\in \overline{B}_{\mathcal{D}^\ast(z_0)} 
\quad (\text{resp.\ } D\psi(z_0) \in \mathbb{R}^n \setminus B_{\mathcal{D}_\ast(z_0)}).
\end{aligned}
\]
\end{definition}

The motivation for introducing this definition lies at the heart of the present work: it reflects the fact that the geometry of the free boundary may vary from point to point. We emphasize that, although the functions $w^-$, $w$, and $w^+$ defined above satisfy the equation in the viscosity sense within their positivity sets, the free boundary condition is guaranteed to hold only for $w^-$ and $w^+$. This observation is summarized in the following result.

\begin{lemma}
Let $w^{-}$ and $w^{+}$ be defined as in \eqref{extremal functions}.  
Then $w^{-}$ satisfies the free boundary condition from below, while $w^{+}$ 
satisfies it from above.
\end{lemma}

\begin{proof}
We prove the result for $w^-$ only. By definition,
$$
    w^- = \varrho(z_0)^{-\frac{1}{\beta(z_0)}} u^{\frac{1}{\beta
    _\ast}},
$$
where $u$ is a minimizer of the functional \eqref{fifi} and $z_0 \in F(u)$. We assume $z_0 = 0$, for simplicity, and argue by compactness. Assume, seeking a contradiction, that there is $\psi \in C^2$ such that $\psi^+$ touches $w^-$ from below at $x_0 \in F(w^-)$, with $|D\psi(x_0)| \not = 0$ such that
$$
    |D\psi(x_0)| > \mathcal{D}^\ast(x_0).
$$
If we do a Lipschitz rescaling on both $\psi$ and $w^-$ around $x_0$, we conclude that the function $\psi_r(x) \coloneqq r^{-1}\psi(x_0 + rx)$ is such that $(\psi_r)_+$ touches $w_r(x) = r^{-1}w^-(x_0 + rx)$ from below at $0 \in F(w_r)$, with $|D\psi_r(0)| > \mathcal{D}^\ast(x_0)$. Recalling the definition of $w^-$, we have
\begin{align*}
    w_r(x) & = \frac{w^-(x_0 + rx)}{r} = \frac{\varrho(0)^{-\frac{1}{\beta(0)}}u^{\frac{1}{\beta_\ast}}(x_0 + rx)}{r}\\
            & = \frac{\varrho(0)^{-\frac{1}{\beta(0)}}}{\varrho(x_0)^{-\frac{1}{\beta(x_0)}}}\frac{\varrho(x_0)^{-\frac{1}{\beta(x_0)}}u^{\frac{1}{\beta_\ast}}(x_0 + rx)}{r}\\
            & = \varphi(x_0,0)\frac{\varrho(x_0)^{-\frac{1}{\beta(x_0)}}u^{\frac{1}{\beta_\ast}}(x_0 + rx)}{r}\\
            & \leq \varphi(x_0,0)\, \varrho(x_0)^{-\frac{1}{\beta(x_0)}} u_r(x)^{\frac{1}{\beta(x_0)}},
\end{align*}
where we used that $1/\beta_\ast \geq 1/\beta(x_0)$ and $u_r(x) \coloneqq r^{-\beta(x_0)}u(x_0+rx)$. By optimal regularity estimates, the sequence $u_r$ converges to a minimizer $u_0$ to the Alt--Philips functional with $\delta \equiv \delta(x_0)$ and $\gamma \equiv \gamma(x_0)$. In summary, we obtain $(\psi_r)_+$ touches 
$$
    \varphi(x_0,0)\, \varrho(x_0)^{-\frac{1}{\beta(x_0)}} \left(u_r(x) \right)^{\frac{1}{\beta(x_0)}}
$$
from below at $x=0$, for every $r>0$, with $|D\psi_r(0)|>\mathcal{D}^\ast(x_0)$. Passing to the limit, we get that $(D\psi(0)\cdot x)_+$ touches 
$$
    \varphi(x_0,0)\, \varrho(x_0)^{-\frac{1}{\beta(x_0)}} u_0^{\frac{1}{\beta(x_0)}}
$$
from below at $x = 0$. However, since $u_0$ is a minimizer of the Alt--Philips functional with constant $\delta \equiv \delta(x_0)$ and $\gamma \equiv \gamma(x_0)$, and $0 \in F(u_0)$, the function
$$
    w_0 \coloneqq \varrho(x_0)^{-\frac{1}{\beta(x_0)}} u_0^{\frac{1}{\beta(x_0)}}
$$
should satisfy the free boundary condition as in Definition \ref{sub-super solution property}, with $\mathcal{D}^\ast = \mathcal{D}_\ast = 1$, see \cite{AP}. But then, the previous reasoning would imply
$$
    \frac{1}{\varphi(x_0,0)}|D\psi(0)| \leq 1,
$$
which is a contradiction, since
$$
    \frac{1}{\varphi(x_0,0)}|D\psi(0)| > \frac{1}{\varphi(x_0,z_0)} \mathcal{D}^\ast(x_0) \geq 1.
$$
The proof that $w^+$ satisfies the free boundary condition from above follows the same reasoning.
\end{proof}

The discussions presented above bring us to the next crucial tool required in the proof of the free boundary regularity. 

\begin{proposition}\label{improvement of flatness}
Let $w$ be a viscosity solution to \eqref{eq almost linearized}, with $0 \in F(w)$, and assume
$$
    \sup\left([\gamma]_{C^{0,\mu}(0)}, [\delta]_{C^{0,\mu}(0)}\right) \leq \epsilon^2.
$$
There exist universal positive parameters $\epsilon_0$ and $r$, such that if
\begin{equation}\label{flat 1}
    (x_n - \epsilon)_+ \leq w^-(x) \leq w(x) \leq w^+(x) \leq (x_n + \epsilon)_+ \quad \text{in }\, B_1,
\end{equation}
with $\epsilon \leq \epsilon_0$, then there exists $\nu \in \partial B_1$ such that
\begin{equation}\label{improv flat 1}
    \left(\nu \cdot x - \frac{\epsilon}{2}r\right)_+ \leq w^-(x) \leq w(x) \leq w^+(x) \leq \left(\nu \cdot x + \frac{\epsilon}{2}r\right)_+ \quad \text{in }\, B_r,
\end{equation}
with $|\nu - e_n|\leq C\epsilon$, where $C$ depends only on universal parameters.
\end{proposition}

\begin{proof}
Assume, seeking a contradiction, that the proposition fails to hold. Then, there would be a sequence $\epsilon_k \to 0$, and a sequence of solutions $w_k$ to \eqref{eq almost linearized} satisfying \eqref{flat 1}, with $\gamma_k$ and $\delta_k$ satisfying
\begin{equation}\label{coefs converging to constant}
    \sup\left([\gamma_k]_{C^{0,\mu}(0)}, [\delta_k]_{C^{0,\mu}(0)}\right) \leq \epsilon_k^2,
\end{equation}
but failing to satisfy \eqref{improv flat 1} for any unit vector $\nu \in \partial B_1$. We define
$$
  v_k \coloneqq \frac{w_k - x_n}{\epsilon_k},  
$$
which is normalized by \eqref{flat 1}. By a consequence of Lemma \ref{almost equicontinuity}, it follows that $v_k$ converges to a function $v_\infty$ locally uniformly in $B_{1/2}\cap \{x_n \geq 0\}$. We further observe that the functions defined by
$$
    v_k^- \coloneqq \frac{w_k^- - x_n}{\epsilon_k} \quad \text{and} \quad v_k^+ \coloneqq \frac{w_k^+ - x_n}{\epsilon_k}
$$
also converge to the same limit $v_\infty$. Recall that since $[\gamma_k]_{C^{0,\mu}(0)} \leq \epsilon_k^2$, we obtain $\gamma_k \to \gamma_\infty$ locally uniformly, where $\gamma_\infty \in (0,1)$ is a constant, and also $\delta_k \to \delta_\infty$, for some constant $\delta_\infty$. Let us now show that $v_\infty$ is a viscosity solution to
\begin{equation}\label{linearized eq DSS}
\left\{
    \begin{array}{rllcl}
        \displaystyle \Delta v_\infty + 2(\beta_\infty-1)\frac{\partial_{e_n}v_\infty}{x_n} & = & 0 & \mathrm{in} & \,B_{1/2}\cap \{x_n>0\}\\
     \partial_{e_n}v_\infty & = & 0 & \mathrm{on} &  B_{1/2}\cap \{x_n = 0\},
    \end{array}
\right.
\end{equation}
in the viscosity sense, where $\beta_\infty \coloneqq 2/(2-\gamma_\infty)$. The key idea is that we can either use $v_k$, $v_k^+$, or $v_k^-$ to approximate $v_\infty$. Indeed, notice that $v_k$ solves
$$
    \Delta v_k = \frac{1}{\epsilon_k}\,\frac{h_k(x,x_n + \epsilon_k v_k,e_n + \epsilon_k Dv_k)}{x_n + \epsilon_k v_k} \eqqcolon \mathcal{R}_k(x),
$$
where $h_k$ is defined as in \eqref{RHS of almost linear eq} with $\gamma_k$ and $\delta_k$ instead. We rewrite $h_k$ as
$$
   h_k(x,x_n + \epsilon_k v_k,e_n + \epsilon_k Dv_k) = g_k\left(x,x_n + \epsilon_k v_k\right) + (\beta_k-1)\left(1-|e_n + \epsilon_k Dv_k|^2\right),
$$
where
$$
    g_k(x,s) \coloneqq \delta_k(x)\frac{\gamma_k(x)}{\beta_k}\varrho_k^{\gamma_k(x)-2}s^{\beta_k(\gamma_k(x)-\gamma_k(0))} - (\beta_k-1),
$$
and
$$
    \varrho_k \coloneqq \left[\frac{(\beta_k(0)-1)\beta_k(0)}{\gamma_k(0) \delta_k(0)} \right]^{\frac{1}{\gamma_k(0) - 2}} \longrightarrow \varrho_\infty \coloneqq \left[\frac{(\beta_\infty-1)\beta_\infty}{\gamma_\infty \delta_\infty} \right]^{\frac{1}{\gamma_\infty - 2}}.
$$
Now we see that
\begin{equation}\label{convergence to zero}
    g_k\left(x,x_n + \epsilon_k v_k\right) = o(\epsilon_k), \quad \text{as }\, k \to \infty.
\end{equation}
Indeed, for $x_n \geq \eta > 0$, we rewrite
$$
    (x_n + \epsilon_k v_k)^{\beta_k(\gamma_k(x)-\gamma_k(0))} = e^{\beta_k(\gamma_k(x)-\gamma_k(0))\ln(x_n + \epsilon_k v_k)}.
$$
For $k$ large enough, we have
$$
    \ln(x_n + \epsilon_k v_k) = \ln(x_n) + o(1),
$$
and so, a Taylor expansion of the exponential near zero gives
$$
    (x_n + \epsilon_k v_k)^{\beta_k(\gamma_k(x)-\gamma_k(0))} = 1 + \beta_k(\gamma_k(x)-\gamma_k(0))\ln(x_n + \epsilon_k v_k) + o(\epsilon_k).
$$
As a consequence,
\begin{align*}
    \frac{g_k(x,x_n + \epsilon_k v_k)}{(\beta_k-1)} & =  \delta_k(x)\frac{\gamma_k(x)}{\beta_k(\beta_k -1)}\varrho_k^{\gamma_k(x)-2}(x_n + \epsilon_k v_k)^{\beta_k(\gamma_k(x)-\gamma_k(0))} - 1\\
    & =  \delta_k(x)\frac{\gamma_k(x)}{\beta_k(\beta_k -1)}\varrho_k^{\gamma_k(x)-2} - 1 + o(\epsilon_k)\\
    & \quad + \delta_k(x)\frac{\gamma_k(x)}{(\beta_k -1)}\varrho_k^{\gamma_k(x)-2}(\gamma_k(x)-\gamma_k(0))\ln(x_n + \epsilon_k v_k) \\
    & = \delta_k(x)\frac{\gamma_k(x)}{\beta_k(\beta_k -1)}\varrho_k^{\gamma_k(x)-2} - 1 + o(\epsilon_k).
\end{align*}
By similar computations, we also get
$$
    \delta_k(x)\frac{\gamma_k(x)}{\beta_k(\beta_k -1)}\varrho_k^{\gamma_k(x)-2} - 1 = o(\epsilon_k),
$$
and so \eqref{convergence to zero} follows. Therefore,
\begin{align*}
    \mathcal{R}_k(x) & =  \frac{1}{x_n + \epsilon_k v_k}\frac{g_k(x,x_n + \epsilon_kv_k)}{\epsilon_k}+ \frac{(\beta_k-1)}{x_n + \epsilon_k v_k}(- 2Dv_k \cdot e_n - \epsilon_k |Dv_k|^2)\\
    & =  \frac{1}{x_n + \epsilon_k v_k}\,o(1) - \frac{2(\beta_k-1)}{x_n + \epsilon_k v_k} \partial_{e_n}v_k + o(1),
\end{align*}
from which the first equation in \eqref{linearized eq DSS} for $v_\infty$ follows. To show that it also solves the second equation, let $y \in \{x_n =0\}$ and assume, seeking a contradiction, that
$$
    \phi(x',x_n) \coloneqq A|x' - y'|^2 + b + px_n^{1-s},
$$
with $s = 2(\beta_\infty -1)$ and $p<0$, touches $v_\infty$ from above at $y$. As in \cite{DSS}, we can replace this test function by
$$
    A|x' - y'|^2 + b - Lx_n^2 + \frac{p}{2}x_n,
$$
for $L>2A$ large enough, which still touches $v_\infty$ from above at $y$. Direct computations give
$$
    \Delta \phi = -2(L - nA) < 0,
$$
for $L$, again, large enough. Now, recalling the convergences discussed at the beginning of the proof, we have $v_k^+ \to v_\infty$ locally uniformly. It then follows that, for 
$$
    \lambda_k \coloneqq \sup(v_k^+ - \phi),
$$
$\phi + \lambda_k$ touches $v_k^+$ from above at $y_k \to y$. Recalling the definition of $v_k^+$, we have that $\phi_k \coloneqq x_n + \epsilon_k(\phi + \lambda_k)$ touches $w_k^+$ from above at $y_k$. On the one hand, we have $\Delta \phi_k < 0$, for every $k \in \mathbb{N}$. On the other hand, since $p<0$, we have
$$
    |D\phi_k(y_k)| \leq \sqrt{1 + cp\epsilon_k} < 1,
$$
for large $k$, and
$$
    \delta_k(y_k)\frac{\gamma_k(y_k)}{\beta_k^\ast}\varrho_k^{\gamma_k(y_k)-2}(\phi_k(y_k))^{\beta_k^\ast(\gamma_k(y_k)-\gamma_k^\ast)} \geq \delta_k(y_k)\frac{\gamma_k(y_k)}{\beta_k^\ast}\varrho_k^{\gamma_k(y_k)-2} \approx 1,
$$
as $k \to \infty$, where we used that $\gamma_k(y_k)-\gamma_k^\ast \leq 0$. As a consequence, $\phi_k$ is a strict supersolution to the equation for $w_k^+$, but this is a contradiction. This gives us that $y_k$ cannot be in the positivity set of $w_k^+$ for large $k$, and thus has to be at the free boundary. The free boundary condition for $w_k^+$ gives
$$
    |D\phi_k(y_k)| \geq \mathcal{D}_\ast(y_k),
$$
where $\mathcal{D}_\ast(y_k)$ is defined with $\delta_k$ and $\gamma_k$ instead. By Assumption \eqref{coefs converging to constant}, we obtain $\mathcal{D}_\ast(y_k) = 1 + O(\epsilon_k^2)$, from which follows that
$$
    1 + O(\epsilon_k^2) \leq |D\phi_k(y_k)| \leq \sqrt{1 + cp\epsilon_k},
$$
which is a contradiction for $k$ large, as $p<0$. This shows that $v_\infty$ is a subsolution to the second equation in \eqref{linearized eq DSS}. To show it is also a supersolution is similar, but using $v_k^-$ instead.

The last step consists of importing regularity from the limiting problem \eqref{linearized eq DSS} back to $v_k$, which is standard. We just point out that we use $v_k^+$ and $v_k^-$ to do so. The $C^{1,\overline{\delta}}$ regularity of \eqref{linearized eq DSS} was obtained in \cite[Theorem 7.2]{DSS}.
\end{proof}

Let us now prove a lemma that gives the equicontinuity of the sequence $v_k$ used in the proof of Proposition \ref{improvement of flatness}. We denote with $B_r'(x')$ the $\mathbb{R}^{n-1}-$di\-mensional ball, centered at $x'$ and radius $r>0$. We omit the center whenever $x' = 0$, and we may also abuse notation and identify $x' = (x',0)$.

\begin{lemma}\label{almost equicontinuity}
Let $w$ be a viscosity solution to \eqref{eq almost linearized} with $0 \in F(w)$, satisfying
\begin{equation}\label{smallness of oscillating}
    \sup\left([\gamma]_{C^{0,\mu}(0)}, [\delta]_{C^{0,\mu}(0)}\right) \leq \epsilon^2,
\end{equation}
and assume
\begin{equation*}
    (x_n)_+ \leq w^-(x) \leq w(x) \leq w^+(x) \leq (x_n + 2 \epsilon)_+ \quad \text{in }\, B_1.
\end{equation*}
There are constants $\epsilon_0>0$ and $r>0$ such that if $\epsilon \leq \epsilon_0$ and
$$
    (\overline{x}_n + \epsilon)_+ \leq w^-(\overline{x}), \quad \text{for} \quad \overline{x} = \frac{1}{5}e_n,
$$
then
$$
    \left(x_n + c \epsilon\right)_+ \leq w^-(x), \quad \text{in} \quad B_r,
$$
for some $c \in (0,1)$ universal. Similarly, if
\begin{equation*}
    (x_n - 2\epsilon)_+ \leq w^-(x) \leq w(x) \leq w^+(x) \leq (x_n)_+ \quad \text{in }\, B_1,
\end{equation*}
and
$$
    (\overline{x}_n - \epsilon)_+ \geq w^+(\overline{x}), \quad \text{for} \quad \overline{x} = \frac{1}{5}e_n,
$$
then
$$
    \left(x_n - c \epsilon\right)_+ \geq w^+(x), \quad \text{in} \quad B_r.
$$
\end{lemma}

\begin{proof}
Consider
$$
    v \coloneqq \frac{w^- - x_n}{\epsilon}.
$$
This function satisfies $0 \leq v \leq 2$ by the flatness assumption. For a fixed small parameter $\vartheta>0$, define the cylinder
$$
    \mathcal{C} \coloneqq B'_{3/4} \times \{\vartheta/2 < x_n < 1/2 \}.
$$
Again, by the assumption of the lemma, it follows that $\mathcal{C} \subset \{w^->0 \}$. Notice that
\begin{align*}
    \Delta v & = \frac{1}{\epsilon} \Delta w^-\\
             & = \frac{1}{\epsilon w^-}\left(\Lambda^-_{0}(x)(w^-)^{\beta_\ast(\gamma(x)-\gamma_\ast)} - (\beta_\ast - 1)|Dw^-|^2 \right)\\
             & = \frac{1}{\epsilon (x_n + \epsilon v)}\left(\Lambda^-_{0}(x)(x_n + \epsilon v)^{\beta_\ast(\gamma(x)-\gamma_\ast)} - (\beta_\ast - 1)|e_n + \epsilon Dv|^2 \right).
\end{align*}
In view of \eqref{smallness of oscillating}, we have
$$
    \Lambda_0^-(x) = (\beta_\ast - 1) + O(\epsilon^2).
$$
Also, since $x_n + \epsilon v \geq \vartheta/2$ inside $\mathcal{C}$ for $\epsilon$ small enough, we obtain
$$
    (x_n + \epsilon v)^{\beta_\ast(\gamma(x)-\gamma_\ast)} = 1+ O(\epsilon^2),
$$
and so
$$
    \Lambda^-_{0}(x)(x_n + \epsilon v)^{\beta_\ast(\gamma(x)-\gamma_\ast)} =(\beta_\ast - 1) + O(\epsilon^2).
$$
This allows us to obtain
\[
\begin{aligned}
|\Delta v|
&\leq \frac{C}{\epsilon}\Big(o(\epsilon) + \epsilon |Dv| + \epsilon^2 |Dv|^2 \Big) \\
&\leq C\big(o(1) + |Dv|\big),
\qquad \text{in } \mathcal{C} \cap \{\,|Dv| \lesssim \epsilon_0^{-1}\,\}.
\end{aligned}
\]
Here, we also used that, in $\mathcal{C} \cap \{|Dv| \lesssim \epsilon_0^{-1}\}$, one has
\[
x_n + \epsilon v \geq \tfrac{\vartheta}{2}
\qquad \text{and} \qquad 
\epsilon |Dv|^2 \leq \epsilon \epsilon_0^{-1} |Dv| \leq |Dv|,
\quad \text{for } \epsilon \leq \epsilon_0.
\] 
Up to a scaling factor, we can still apply \cite[Lemma 3.8]{DSS} to get that $v$ is locally Lipschitz in the interior of the cylinder $\mathcal{C}$. Since $v$ is nonnegative and satisfies $v \geq 1$ at $\overline{x} = e_n/5$, we can apply the Harnack inequality, provided $\epsilon_0$ is sufficiently small. Recall that $o(1) \to 0$ as $\epsilon_0 \to 0$. This yields $v \geq c$ in $B_{1/2}' \times \{x_n = \vartheta\}$, which in turn implies that 
\begin{equation}\label{bound from below harnack}
	w^- \geq x_n + \epsilon c, \text{ in }\, B_{1/2}' \times \{x_n = \vartheta\}.
\end{equation}
Now, we extend this inequality beyond $\{x_n = \vartheta\}$ by building suitable barriers. Consider
$$
    B(x) \coloneqq -|x|^2 + Ax_n^2 + x_n, 
$$
with $A>0$ large enough, and define, for $t \in \mathbb{R}$,
$$
    \Phi_t(x) \coloneqq x_n + \epsilon_0 c(B+t).
$$
First, we observe that
$$
	B(x) + t \leq (A-1)x_n^2 + x_n + t \leq Ax_n + t,
$$
and so, if $t$ is negative enough, depending only on $A$, we have
$$
	\Phi_t(x) < x_n \leq w^-,\, \text{ in }\, \mathcal{L} \coloneqq \overline{B}_{1/2}' \times \{-2\epsilon_0 \leq x_n \leq \vartheta\}.
$$
This allows us to slide this barrier $\Phi_t$ until it touches the graph of $w^-$ from below for the first time, that is, we consider $t'$ to be the largest value of $t$ such that
$$
    \Phi_t \leq w^- \quad \text{in} \quad B'_{1/2} \times \{-2\epsilon_0 \leq x_n \leq \vartheta\}.
$$
By maximality, there should be $x'$ such that $\Phi_{t'}(x') = w^-(x')$. We observe that $t' \geq 1/8$ should hold; otherwise, if $t' < 1/8$, we would get $\Phi_{t'} < w^-$ on the boundary of the cylinder $B'_{1/2} \times \{-2\epsilon_0 \leq x_n \leq \vartheta\}$. Indeed, if $x_n = -2\epsilon_0$, then
$$
	\Phi_{t'}(x) \leq -2\epsilon_0 + \epsilon_0 c\left(-4\epsilon_0^2 + 4A\epsilon_0^2 - 2\epsilon_0 + 1/8\right) < 0 = w^-,
$$
for $\epsilon_0$ small enough. Recall that $w^-=0$ on $x_n = -2\epsilon_0$ by the flatness assumption. If $x_n = \vartheta$, then
$$
	\Phi_{t'}(x) \leq \vartheta + \epsilon_0 c\left(-\vartheta^2 + A\vartheta^2 - \vartheta + 1/8\right) < \vartheta + \epsilon_0 c,
$$
for $\vartheta$ small enough. Therefore, $\Phi_{t'} < x_n + \epsilon_0 c \leq w^-$ on $B_{1/2}' \times \{x_n = \vartheta\}$ by taking \eqref{bound from below harnack} into account. Moreover, on $\partial B_{1/2}' \times \{-2\epsilon_0 \leq x_n \leq \vartheta\}$, we have
$$
	\Phi_{t'}(x) = x_n + \epsilon_0 c\left(-1/4 - x_n^2 + Ax_n^2 + x_n + 1/8\right) < x_n,
$$
provided $\epsilon_0$ and $\vartheta$ are small enough. As a consequence, we get that the touching occurs in the interior of the cylinder $\mathcal{L}$ and has to be either in the positivity set of $\Phi_{t'}$ or in its free boundary. Assume, seeking a contradiction, that $x' \in F(w^-)$. Then, there should hold
$$
    |D\Phi_{t'}(x')| \leq \mathcal{D}^\ast(x').
$$
On the other hand, by direct computations, we have
\begin{align}
\nonumber |D\Phi_{t'}(x')|^2
  &= \bigl| e_n + \epsilon c\, DB(x') \bigr|^2 \\ \nonumber
  &= 1 + 2\epsilon c\, \partial_{e_n} B(x') + \epsilon^2 c^2 |D B(x')|^2 \\
  &> 1 + 2\varepsilon c\,\bigl(2(A-1)x_n + 1\bigr), \label{eq:inequality}
\end{align}
which implies
$$
    |D\Phi_{t'}(x')| > 1 + c_1\sqrt{\epsilon},
$$
for $\epsilon_0$ small enough. However, by \ref{smallness of oscillating}, it follows that
$$
   \mathcal{D}^\ast(x') \leq 1 + O(\epsilon^2), 
$$
which is a contradiction for $\epsilon$ small enough. It also cannot occur in the interior because $\Phi_{t'}$ is a strict subsolution. Indeed, observe that $\Delta \Phi_{t'} = \epsilon c \Delta B > 0$, and as a consequence of \eqref{eq:inequality}, we have
\begin{align*}
   \Lambda^-_{z_0}(x)\left(\Phi_{t'}(x')\right)^{\beta_\ast(\gamma(x)-\gamma_\ast)} - (\beta_\ast - 1)|D\Phi_{t'}(x')|^2 & <   \Lambda^-_{z_0}(x) - (\beta_\ast - 1) - \epsilon c_2\\
   & \leq  0,
\end{align*}
where we used that $\gamma(x)-\gamma_\ast \geq 0$, and the parameters $\epsilon_0$ and $\mu$ are small enough so that 
$$
\Lambda^-_{z_0}(x) - (\beta_\ast - 1) < \frac{1}{2}\epsilon c_2.
$$
Therefore, $\Phi_{t'}$ is a strict subsolution at $x'$, which is a contradiction.

Now we use the bound from below $t' \geq 1/8$ to obtain
$$
    w^- \geq \Phi_{t'} \geq x_n + \epsilon_0 c(B + 1/8),
$$
in the cylinder $B_{1/2}' \times \{-2\epsilon_0 \leq x_n \leq \vartheta\}$. In this set, we have $B \geq -1/16$, and so 
$$
    w^- \geq \Phi_{t'} \geq x_n + \epsilon_0 c_3,
$$
in the same cylinder, and the proof of the lemma follows. 

\medskip
For the second part, our goal is to improve oscillation from above. To that end, we define instead
$$
    v \coloneqq \frac{x_n - w^+}{\epsilon}.
$$
By the flatness assumption, it satisfies $0 \leq v \leq 2$. Consider the cylinder $\mathcal{C}$ as before. If $\epsilon_0 \leq \vartheta/8$, then the flatness assumption gives us $\mathcal{C} \subset \{ w^+>0\}$. As before, $v$ solves
$$
    \Delta v = \frac{1}{\epsilon (x_n - \epsilon v)}\left(\Lambda^+_{0}(x)(x_n - \epsilon v)^{\beta^\ast(\gamma(x)-\gamma^\ast)} - (\beta^\ast - 1)|e_n - \epsilon Dv|^2 \right).
$$
Using that $x_n - \epsilon v \geq \vartheta/4$ in $\mathcal{C}$, we get
$$
    |\Delta v| \leq C(o(1) + |Dv|), \quad \text{in} \quad \mathcal{C} \cap \{|Dv| \lesssim \epsilon_0^{-1}\}.
$$
As before, we combine the Harnack inequality and $v \geq 1$ at $\overline{x} = e_n/5$ to obtain $v \geq c$ in $B_{1/2}' \times \{x_n = \vartheta \}$, which then implies
\begin{equation}\label{bound from above harnack}
    w^+ \leq x_n - \epsilon c, \quad \text{in} \quad B_{1/2}' \times \{x_n = \vartheta \}.
\end{equation}
To extend this inequality beyond $x_n = \vartheta$, we consider the following barrier
$$
    B(x) \coloneqq |x|^2 - Ax_n^2 - x_n,
$$
with $A>0$ and define for $t>0$
$$
    \Phi_t(x) \coloneqq x_n + \epsilon_0 c(B(x) + t).
$$
First notice that if $A > n$, then
$$
    \Delta \Phi_t = \epsilon_0 c \Delta B < 0.
$$
Moreover, by the very same computations
$$
    |D\Phi_t|^2  < 1 + O(\epsilon),
$$
and so
$$
    \Lambda^+_{z_0}(x)(\Phi_t)^{\beta^\ast(\gamma(x)-\gamma^\ast)} - (\beta^\ast - 1)|D\Phi_t|^2 \geq 0,
$$
so that $\Phi_t$ is a strict supersolution for any $t>0$. It is used here that $\gamma(x) - \gamma^\ast \leq 0$. The rest of the proof follows as before, by sliding $\Phi_t$ from above until it touches $w^+$ for the first time and estimating the $t'$ from above.
\end{proof}

It is fairly standard to verify that Lemma \ref{almost equicontinuity} yields equicontinuity for the family of renormalized solutions; see, for instance, \cite{DSS}.

\begin{proposition}\label{mco airport}
Let $w$ be a viscosity solution to \eqref{eq almost linearized} with $0 \in F(w)$. There exists $\epsilon_0$ such that if $\epsilon \leq \epsilon_0$ and there holds
$$
    \sup\left([\gamma]_{C^{0,\mu}(0)}, [\delta]_{C^{0,\mu}(0)}\right) \leq \epsilon^2,
$$ 
and
$$
    (x \cdot \nu - \epsilon)_+ \leq w^- \leq w \leq w^+ \leq (x \cdot \nu + \epsilon)_+ \quad \text{in} \quad B_1,
$$
for some $\nu \in \partial B_1$, then $F(w)$ is $C^{1,\delta}$ at $0$, for some universal parameter $\delta>0$.
\end{proposition}

\begin{proof}
After a rotation, we can apply Proposition \ref{improvement of flatness} to obtain $\nu_1 \in \partial B_1$ satisfying $|\nu_1 - \nu| \leq C\epsilon$ and
$$
    \left(\nu_1 \cdot x - \frac{\epsilon}{2}r\right)_+ \leq w^- \leq w \leq w^+ \leq \left(\nu_1 \cdot x + \frac{\epsilon}{2}r\right)_+ \quad \text{in }\, B_r.
$$ 
Defining $w_r(x) \coloneqq r^{-1}w(rx)$, this can be written as
$$
    \left(\nu_1 \cdot x - \frac{\epsilon}{2}\right)_+ \leq (w^-)_r \leq w_r \leq (w^+)_r \leq \left(\nu_1 \cdot x + \frac{\epsilon}{2}\right)_+ \quad \text{in }\, B_1,
$$
where
$$
   (w^-)_r(x) \coloneqq r^{-1}w^-(rx) \quad \text{and} \quad  (w^+)_r(x) \coloneqq r^{-1}w^+(rx).
$$
Moreover, since
$$
  \gamma_\ast(0,1) \leq \gamma_\ast(0,r) \leq \gamma(0) \leq \gamma^\ast(0,r) \leq \gamma^\ast(0,1),  
$$
it follows that
$$
    u^{\frac{1}{\beta_\ast(0,1)}} \leq u^{\frac{1}{\beta_\ast(0,r)}} \leq u^{\frac{1}{\beta(0)}} \leq u^{\frac{1}{\beta^\ast(0,r)}} \leq u^{\frac{1}{\beta^\ast(0,1)}}.
$$
This allows us to further squeeze the flatness inequality by
$$
    \left(\nu_1 \cdot x - \frac{\epsilon}{2}\right)_+ \leq w^-_r \leq w_r \leq w^+_r \leq \left(\nu_1 \cdot x + \frac{\epsilon}{2}\right)_+ \quad \text{in }\, B_1,
$$
where
$$w^-_r(x) \coloneqq \varrho(0)^{-\frac{1}{\beta(0)}} r^{-1}u(rx)^{\frac{1}{\beta_\ast(0,r)}}$$ 
and
$$w^+_r(x) \coloneqq \varrho(0)^{-\frac{1}{\beta(0)}} r^{-1}u(rx)^{\frac{1}{\beta^\ast(0,r)}}. $$
We can then apply again Proposition \ref{improvement of flatness}, this time to $w_r$, to obtain $\nu_2 \in \partial B_1$ such that $|\nu_2 - \nu_1| \leq C2^{-1}\epsilon$, and 
$$
    \left(\nu_2 \cdot x - \frac{\epsilon}{4}r\right)_+ \leq w^-_r \leq w_r \leq w^+_r \leq \left(\nu_2 \cdot x + \frac{\epsilon}{4}r\right)_+ \quad \text{in } \quad B_r.
$$
An iteration of this argument gives a sequence $(\nu_k)_{k \in \mathbb{N}} \subset \partial B_1$ satisfying $|\nu_k - \nu_{k-1}| \leq C 2^{-k}\epsilon$, such that
$$
    \left(\nu_k \cdot x - \frac{\epsilon}{2^k}r\right)_+ \leq w^-_{r^{k-1}} \leq w_{r^{k-1}} \leq w^+_{r^{k-1}} \leq \left(\nu_k \cdot x + \frac{\epsilon}{2^k}r\right)_+ \quad \text{in} \quad B_{r^{k-1}}.
$$
Scaling back, we obtain
$$
    \left(\nu_k\cdot x - \frac{\epsilon}{2^k}r^k\right)_+ \leq w \leq \left(\nu_k\cdot x + \frac{\epsilon}{2^k} r^k\right)_+ \quad \text{in} \quad B_{r^k},
$$
for every $k \in \mathbb{N}$. It then follows that $F(w)$ is $C^{1,\delta}$ at $0$.
\end{proof}

We conclude this section with the regularity result for the free boundary. We say a function belongs to $W^{1,n^{+}}$ if it belongs to $W^{1,q}$, for some $q>n$.

\begin{theorem}\label{free boundary regularity theorem}
Let $u$ be a local minimizer of \eqref{fifi} with assumptions \eqref{H1} and \eqref{bound delta} in force, and assume 
$$
    \gamma(x), \delta(x)  \in W^{1,n^{+}}.
$$
Then, the free boundary $F(u)$ is locally a $C^{1,\delta}$ surface, up to a negligible singular set of Hausdorff dimension less than or equal to $n-2$.
\end{theorem}

\begin{proof}
With all the ingredients from the preceding discussion available, the proof is standard, and we only highlight the main steps. 
    
We start by decomposing the free boundary as the disjoint union of its regular points and its singular points, that is,
$$
    F(u) = \reg(u) \cup \sing(u).
$$
The set $\reg(u)$ stands for the points where blow-ups can be classified. More precisely, $z_0 \in \reg(u)$, if for a sequence of radii $r_n$ converging to zero and a unitary vector $\nu$, there holds 
$$
     u_{r_n}(x) \coloneqq \frac{u(z_0 + r_nx)}{r_n^{\frac{2}{2-\gamma(z_0)}}} \longrightarrow \varrho(z_0) (x\cdot \nu)_{+}^{\frac{2}{2-\gamma(z_0)}}.
$$
The set $\sing(u)$ is simply the complement of $\reg(u)$, that is 
$$
    \sing(u) \coloneqq F(u) \setminus \reg(u).
$$
A successful application of the dimension reduction argument ensures that
$$
    \mathcal{H}^{n-2+s}(\sing(u)) = 0,
$$
for all $s>0$. Thus, one can estimate the Hausdorff dimension of the singular set as
$$
    \dim_\mathcal{H}\left(\sing(u)\right) \coloneqq \inf\{d: \mathcal{H}^d(\sing(u)) = 0  \} \leq n-2+s,
$$
for every $s>0$, and so
$$
    \dim_\mathcal{H}\left(\sing(u)\right) \leq n-2.
$$
In particular, we conclude that $\sing(u)$ is a negligible set with respect to the Hausdorff measure $\mathcal{H}^{n-1}$, \textit{i.e.}, 
$$
    \mathcal{H}^{n-1} \left (F(u) \setminus \reg(u) \right ) = 0.
$$
This, in particular, allows us to conclude that the portion of the free boundary to which Proposition \ref{mco airport} can be applied has total measure. 

Let us explain the dimension reduction argument, based on ideas from \cite{DSS2}. We emphasize that this part comes almost for free, after establishing the Weiss monotonicity formula and the classification of blow-ups as homogeneous functions. The reason behind this is that the minimal cones are minimizers of the Alt--Phillips functionals with constant $\delta$ and $\gamma$, for which the classical theory applies.

Fix $z_0 \in F(u)\cap B_{1/2}$ and consider a blow-up of $u$ at $z_0$, namely
\[
    u_r(x) \coloneqq r^{-\beta(z_0)}\,u(z_0+rx) \;\longrightarrow\; U_{z_0}
    \quad\text{in } C^{\frac{2}{\,2-\gamma(z_0)\,}}_{\mathrm{loc}}(\mathbb{R}^n),
\]
along a sequence $r_k \downarrow 0$ (see Remark~\ref{remark de convergencia}). 
By Corollary~\ref{classification of blow ups}, $U_{z_0}$ is a $\beta(z_0)$-minimal cone, \textit{i.e.}, a $\beta(z_0)$-homogeneous global minimizer of the Alt--Phillips functional with 
$\delta \equiv \delta(z_0)$ and $\gamma \equiv \gamma(z_0)$ (cf. Definition~\ref{def: min cone}).

Let $Z_0 \in F(U_{z_0})$. By homogeneity and a rotation/translation, we may assume $Z_0=e_1$ (the first coordinate vector). Let $\overline{U}_{Z_0}$ denote a blow-up of $U_{z_0}$ at $Z_0$. Then $\overline{U}_{Z_0}$ is constant in the $e_1$-direction (by the homogeneity of $U_{z_0}$; see \cite[Lemma~5.4]{DSS0}) and is again a global minimizer of the same Alt--Phillips functional. Since $\overline{U}_{Z_0}$ is independent of the first variable, it can be viewed as a minimizer in $\mathbb{R}^{n-1}$ (see \cite[Proposition~3.13]{DSS0} and \cite[Proposition~5.3]{DSS2}). If $Z_0$ is a singular point of $U_{z_0}$, then $\overline{U}_{Z_0}$ is a nontrivial cone.

Iterating this dimension-reduction argument, we eventually reach dimension $2$, where minimizing cones are classified and hence no singular points occur (see \cite[Theorem~8.2]{AP}). In particular, this yields the existence of a $\beta(z_0)$-minimal cone in $\mathbb{R}^3$ that is regular at every point except the origin. We also point out that an energy gap also follows from the Weiss monotonicity formula.

The main part of the argument consists of proving the following claim: assume that, for some $s>0$, one has $\mathcal{H}^s(\sing(U))=0$, for every $\beta(z_0)$-minimal cone $U$ in $\mathbb{R}^n$; then $\mathcal{H}^s(\sing(u))=0$, for every minimizer of \eqref{fifi} in $B_1$. Indeed, let $u$ be a minimizer of \eqref{fifi} in $B_1$ and fix $z\in\sing(u)$. There exists $d_z>0$ such that, for every $\eta \in(0,d_z]$ and every $S\subset \sing(u)\cap B_\eta(z)$, one can cover $S$ by finitely many balls $B_{r_i}(z_i)$, with $z_i\in S$ and
\begin{equation}\label{cover sum}
\sum_i r_i^s \le 2^{-1}\eta^s.
\end{equation}
If not, take $\eta_k\downarrow 0$ for which \eqref{cover sum} fails and consider the blow-ups
$$
    u_k(x)\coloneqq \eta_k^{-\beta(z)}\,u(z+\eta_k x).
$$
Up to a subsequence, $u_k \to U_z$ locally, where $U_z$ is a $\beta(z)$-minimal cone. 
By the hypothesis, $\mathcal{H}^s(\sing(U_z))=0$, so, locally, $\sing(U_z)$ admits a covering 
$\{B_{r_i/4}(z_i)\}$, with $\sum_i r_i^s \le \tfrac12$. 
By Proposition~\ref{mco airport} (flatness implies regularity result), this yields a local covering of 
$\sing(u_k)$ by $\{B_{r_i/2}(z_i)\}$; scaling back gives \eqref{cover sum} for $u$, a contradiction. The remaining measure-theoretic step is identical to \cite[Proof of Lemma~5.7]{DSS0}, which concludes $\mathcal{H}^s(\sing(u))=0$. With the previous claim in hand, the final step in the dimension–reduction scheme is the corresponding statement for minimal cones: assume that for some $s>0$,
$\mathcal{H}^s(\sing(U))=0$ for every minimal cone $U$; then,
\[
\begin{aligned}
&\forall\, U \text{ minimal cone in } \mathbb{R}^n\colon \quad \mathcal{H}^s(\sing( U))=0 \\
&\qquad \Longrightarrow\quad 
\forall\, V \text{ minimal cone in } \mathbb{R}^{n+1}\colon \quad \mathcal{H}^{s+1}(\sing(V))=0 .
\end{aligned}
\] 
The proof follows an analogous reasoning. Since minimal cones are classified in dimension $2$, we know that for all $s>0$,
\[
\begin{aligned}
&\forall\, U \text{ min. cone in } \mathbb{R}^2\colon \quad \mathcal{H}^s(\sing( U))=0 \\
&\qquad \Longrightarrow\quad 
\forall\, V \text{ min. cone in } \mathbb{R}^{3}\colon \quad \mathcal{H}^{s+1}(\sing(V))=0\\
& \quad \qquad \cdots \\
& \qquad \qquad \Longrightarrow\quad \forall\, W \text{ min. cone in } \mathbb{R}^n\colon \quad \mathcal{H}^{s+n-2}(\sing(W))=0.
\end{aligned}
\]
Thus,
\[
\mathcal{H}^{s+n-2}(\sing(u))=0, \quad \text{for all } s>0.
\]

Now, we show that $\reg(u)$ is locally $C^{1,\delta}$, for some $\delta>0$ universal. Consider $z_0 \in \reg(u)$ and let $u_0$ be a blow-up limit of $u$ at $z_0$. In other words, for a sequence $r=o(1)$, and  up to a change of coordinates, there holds
$$
    u_r(x) \coloneqq \frac{u(z_0 + rx)}{r^{\frac{2}{2-\gamma(z_0)}}} \longrightarrow \varrho(z_0) (x_n)_+^{\frac{2}{2-\gamma(z_0)}},
$$
in the  $C_{loc}^{1,\frac{\gamma(z_0)}{2-\gamma(z_0)}}(\mathbb{R}^n)-$topology. Defining further
$$
	u^-_r(x) \coloneqq \frac{u(z_0 + rx)}{r^{\frac{2}{2-\gamma_\ast(z_0,r)}}} \quad \text{and} \quad u^+_r(x) \coloneqq \frac{u(z_0 + rx)}{r^{\frac{2}{2-\gamma^\ast(z_0,r)}}},
$$
we see that they also converge to $\varrho(z_0) (x_n)_+^{\frac{2}{2-\gamma(z_0)}}$, which follows from the fact that
$$
	\lim_{r \to 0} r^{\frac{2-\gamma^\ast(z_0,r)}{2-\gamma_\ast(z_0,r)}} = \lim_{r \to 0} r^{\frac{2-\gamma_\ast(z_0,r)}{2-\gamma^\ast(z_0,r)}} = 1.
$$
Recalling the notation at the beginning of this Section, we have
$$
	w^-_r(x) \coloneqq \varrho(z_0)^{\frac{-1}{\beta(z_0)}}\frac{u(z_0 + rx)^{\frac{1}{\beta_\ast(z_0,r)}}}{r} \longrightarrow (x_n)_+,
$$
and
$$
	w^+_r(x) \coloneqq \varrho(z_0)^{\frac{-1}{\beta(z_0)}}\frac{u(z_0 + rx)^{\frac{1}{\beta^\ast(z_0,r)}}}{r} \longrightarrow (x_n)_+.
$$
Therefore, if we define
$$
	w(x) \coloneqq \varrho(z_0)^{\frac{-1}{\beta(z_0)}}\frac{u(z_0 + rx)^{\frac{1}{\beta(z_0)}}}{r},
$$
we obtain $w^-_r \leq w \leq w^+_r$ in $B_1$. From here, it is then standard to derive the assumptions of Proposition \ref{improvement of flatness} for $r$ small enough, and we conclude $F(u)$ is $C^{1,\delta}$ at $z_0$.
\end{proof}

\bigskip

{\small \noindent{\bf Acknowledgments.} We are deeply grateful to the referee for an exceptionally careful and insightful report. The comments and suggestions were of great help and substantially improved the manuscript. This publication is based upon work supported by King Abdullah University of Science and Technology (KAUST) under Award No. ORFS-CRG12-2024-6430. It was financed in part by the Coordenação de Aperfeiçoamento de Pessoal de Nível Superior - Brasil (CAPES) - Finance Code 001. DJA is supported by CNPq grant 427070/2016-3 and grant 2019/0014 from Para\'\i ba State Research Foundation (FAPESQ). JMU is partially supported by UID/00324 - Centre for Mathematics of the University of Coimbra.}

\bigskip

\bibliographystyle{amsplain, amsalpha}

\begin{thebibliography}{99} 

\bibitem{ABF} E. Acerbi,  G. Bouchitt\'e and I. Fonseca, 
{\it Relaxation of convex functionals: the gap problem}, 
Ann. Inst. H. Poincar\'e Anal. Non Lin\'eaire 20 (2003), 359--390. 

\bibitem{AC} H.W. Alt and L. A. Caffarelli,  
{\it Existence and regularity for a minimum problem with free boundary},
J. Reine Angew. Math. 325 (1981), 105--144.

\bibitem{AP} H.W. Alt and D. Phillips, 
{\it A free boundary problem for semilinear elliptic equations},  
J. Reine Angew. Math. 368 (1986), 63--107.

\bibitem{APPT} P. Andrade, D. Pellegrino, E. Pimentel and E.V. Teixeira,
{\it $C^1$-regularity for degenerate diffusion equations},
Adv. Math. 409 (2022), Paper No. 108667, 34 pp.

\bibitem{AT} D.J. Ara\'ujo and E.V. Teixeira, 
{\it Geometric approach to nonvariational singular elliptic equations}, 
Arch. Ration. Mech. Anal. 209 (2013), no. 3, 1019--1054.

\bibitem{Aris} R. Aris, 
\textit{The mathematical theory of diffusion and reaction}, 
Clarendon Press, Oxford, 1975.

\bibitem{BN} C.M. Brauner and B. Nicolaenko,
\textit{Free boundary value problems as singular limits of nonlinear eigenvalue problems}, 
in: Free boundary problems, Vol. II (Pavia, 1979), pp. 61--84, Istituto Nazionale di Alta Matematica Francesco Severi, Rome, 1980.

\bibitem{BPRT} A.C. Bronzi, E.A. Pimentel, G.C. Rampasso and E.V. Teixeira, 
{\it Regularity of solutions to a class of variable-exponent fully nonlinear elliptic equations},   
J. Funct. Anal. 279 (2020), no. 12, 108781, 31 pp. 

\bibitem{DS} D. De Silva, 
{\it Free boundary regularity for a problem with right hand side}, 
Interfaces Free Bound. 13 (2011), no. 2, 223--238. 

\bibitem{DSS0} D. De Silva and O. Savin, 
{\it Regularity of Lipschitz free boundaries for the thin one-phase problem}, 
J. Eur. Math. Soc. 17 (2015), no. 6, 1293--1326.

\bibitem{DSS} D. De Silva and O. Savin, 
{\it On certain degenerate one-phase free boundary problems}, 
SIAM J. Math. Anal. 53 (2021), no. 1, 649--680. 

\bibitem{DSS2} D. De Silva and O. Savin, 
{\it The Alt--Philips functional for negative powers}, 
Bull. London Math. Soc. 55 (2023), 2749--2777. 

\bibitem{DKV} S. Dipierro, A. Karakhanyan and E. Valdinoci,
{\it Classification of global solutions of a free boundary problem in the plane}, 
Interfaces Free Bound. 25 (2023), no.3, 455--490.

\bibitem{ES} L. El Hajj and H. Shahgholian,
{\it Radial symmetry for an elliptic PDE with a free boundary},
Proc. Amer. Math. Soc. Ser. B 8 (2021), 311–-319.

\bibitem{F} H. Federer, 
{\it The singular sets of area minimizing rectifiable currents with codimension one and of area minimizing flat chains modulo two with arbitrary codimension},  
Bull. Amer. Math. Soc. 76 (1970), 767--771.

\bibitem{EG} E. Giusti, 
\textit{Direct methods in the calculus of variations}, 
World Scientific, Singapore, 2003.

\bibitem{GG} M. Giaquinta and E. Giusti, 
{\it Differentiability of minima of non-differentiable functionals}, 
Invent. Math. 72 (1983), 285--298.

\bibitem{QF} Q. Han, Q. and F. Lin, 
\textit{Elliptic partial differential equations}, 
Courant Lecture Notes in Mathematics, Vol. 1, New York University, 1997.

\bibitem{LQT}  R. Leitão, O. de Queiroz and E.V. Teixeira, 
{\it Regularity for degenerate two-phase free boundary problems}, 
Ann. Inst. H. Poincaré Anal. Non Linéaire 32 (2015), 741--762.

\bibitem{PT24} D. Pellegrino and E. Teixeira, 
\textit{Regularity for quasi-minima of the Alt-Caffarelli functional}, 
Calc. Var. Partial Differential Equations 63 (2024), no.~6, Paper No. 149, 12 pp.

\bibitem{PSU} A. Petrosyan, H. Shahgholian and N. Uraltseva, 
{\it Regularity of free boundaries in obstacle-type problems}, 
Graduate Studies in Mathematics 136, American Mathematical Society, Providence, RI, 2012. 

\bibitem{P1} D. Phillips, 
{\it A minimization problem and the regularity of solutions in the presence of a free boundary}, 
Indiana Univ. Math. J. 32 (1983), no. 1, 1--17.

\bibitem{P2} D. Phillips, 
{\it Hausdorff measure estimates of a free boundary for a minimum problem},
Comm. Partial Differential Equations 8 (1983), no. 13, 1409--1454.

\bibitem{ST} N. Soave and S. Terracini,
{\it The nodal set of solutions to some elliptic problems: singular nonlinearities}, 
J. Math. Pures Appl. (9) 128 (2019), 264--296.

\bibitem{Teix} E. Teixeira,
{\it Nonlinear elliptic equations with mixed singularities},
Potential Anal. 48 (2018), no.3, 325--335.

\bibitem{WY} Y. Wu and H. Yu,
{\it On the fully nonlinear Alt--Phillips equation}, 
Int. Math. Res. Not. IMRN (2022), no.11, 8540--8570.

\bibitem{Y} R. Yang, 
{\it Optimal regularity and nondegeneracy of a free boundary problem related to the fractional Laplacian}, 
Arch. Ration. Mech. Anal. 208 (2013), no.3, 693--723.

\bibitem{Z} L. Zaj\'i\v cek, 
{\it Porosity and $\sigma$-porosity}, 
Real Anal. Exchange 13 (1987/88), no.~2, 314--350.

\end{thebibliography}

\end{document}